\documentclass[reqno]{amsart}
\usepackage{amsfonts}
\usepackage{amsmath}
\usepackage{amsthm, amscd}
\usepackage{mathtools}
\usepackage{amssymb}
\usepackage{mathrsfs}

\allowdisplaybreaks

\theoremstyle{plain}\newtheorem{Theorem}{Theorem}[section]
\theoremstyle{plain}\newtheorem{Corollary}[Theorem]{Corollary}
\theoremstyle{plain}\newtheorem{Lemma}[Theorem]{Lemma}
\theoremstyle{plain}\newtheorem{Definition}[Theorem]{Definition}
\theoremstyle{plain}\newtheorem{Proposition}[Theorem]{Proposition}
\theoremstyle{plain}\newtheorem{Assumption}[Theorem]{Assumption}
\theoremstyle{plain}\newtheorem*{Theorem*}{Theorem}
\theoremstyle{remark}\newtheorem{remark}[Theorem]{Remark}
\theoremstyle{remark}\newtheorem*{remark*}{Remark}

\newcommand{\Hom}{\text{Hom}}
\newcommand{\id}{\text{id}}
\newcommand{\dist}{\text{dist}}
\newcommand{\supp}{\text{supp}}
\newcommand{\grad}{\text{grad}}

\address {Department of Mathematics, Harvard University, Cambridge, MA 02138, USA}                                                   
\email {bzhang@math.harvard.edu}

\title[Rectifiability for the zero loci of $\mathbb{Z}/2$ harmonic spinors]{Rectifiability and Minkowski bounds for the zero loci of $\mathbb{Z}/2$ harmonic spinors in dimension $4$}
\author{Boyu Zhang}
\begin{document}

\begin{abstract}
This article proves that the zero locus of a $\mathbb{Z}/2$ harmonic spinor on a $4$ dimensional manifold is $2$-rectifiable and has locally finite Minkowski content.
\end{abstract}

\maketitle

\section{Introduction}
\subsection{Background}
The notion of $\mathbb{Z}/2$ harmonic spinors was first introduced by Taubes \cite{taubes2012psl, taubesZ2} to describe the behaviour of certain non-convergent sequences of flat $PSL_2(\mathbb{C})$ connections on a three manifold. It also appears in the compactifications of the moduli spaces of solutions to Kapustin-Witten equations \cite{taubes2013compactness}, Vafa-Witten equations \cite{taubes2017behavior}, and Seiberg-Witten equations with multiple spinors \cite{haydys2015compactness, taubes2016behavior}.
These equations may have important topological applications. For example, Witten \cite{witten2014two} has conjectured that the space of solutions to the Kapustin-Witten equations can be used to compute the Jones polynomials and the Khovanov homology for knots. Haydys \cite{haydys2017g2} conjectured a relation between the multiple spinor Seiberg-Witten monopoles, Fueter sections, and $G2$ instantons. More recently, Doan and Walpuski \cite{doan2017counting} conjectured a relation between generalized Seiberg-Witten equations and counting of associative manifolds on $G2$ manifolds.

All of these applications require better understanding of the compactifications for the relevant moduli spaces. The zero locus of $\mathbb{Z}/2$ harmonic spinor plays a crucial role in the description of the boundaries of the compactifications. It is the set of points where the sequence of solutions blow up after normalizations. Takahashi \cite{takahashi2015moduli, takahashi2017index}
studied the moduli spaces of $\mathbb{Z}/2$ harmonic spinors with additional regularity assumptions on the zero locus, where the zero locus was assumed to be a union of embedded circles in the case of dimension $3$, and an embedded surface in the case of dimension $4$. In general, the zero locus may not have this regularity. Taubes \cite{taubesZ2} proved that the zero locus must have Hausdorff codimension at least $2$. This article improves the regularity result by proving that the zero locus is rectifiable and has locally finite Minkowski content. The arguments are inspired by \cite{rec}, where a similar problem was studied for Dir-minimizing $Q$-valued functions.  The proof relies on a general method developed recently by Naber and Valtorta \cite{naber2015rectifiable}.

\subsection{Statement of results}
Let $X$ be a 4-dimensional Riemannian manifold. Let $\mathcal{V}$ be a Clifford bundle over $X$. That is, $\mathcal{V}$ is a unitary vector bundle equipped with an extra structure $\rho\in \Hom(TX,\Hom(\mathcal{V},\mathcal{V}))$, such that $\rho(e)^2=-\|e\|^2\cdot\id$ and $\|\rho(e)(u)\|=\|e\|\cdot\|u\|$ for every $e\in T_pX$ and $u\in\mathcal{V}|_p$. Let $\nabla$ be a connection on $V$ that is compatible with $(X,\mathcal{V},\rho)$. Namely, for every pair of smooth vector fields $e$, $e'$, and every smooth section $u$ of $\mathcal{V}$, one has
$$
\nabla_e(\rho(e')\cdot u)=\rho(\nabla_e e') \cdot u + \rho(e')\cdot\nabla_e(u).
$$
The Dirac operator on $\mathcal{V}$ is defined by
$$
D(u)=\sum_{i=1}^4 \rho(e_i) \nabla_{e_i} u,
$$
where $\{e_i\}$ is a local orthonormal frame for $TX$.

Let $Q$ be a positive integer. For a vector space $E$, define $\mathcal{A}_Q(E)$ to be the set of unordered $Q$-tuples of points in $E$. If $P_1, P_2,\cdots, P_Q$ are $Q$ points in $E$, use $\sum_{i=1}^Q [\![P_i]\!]\in\mathcal{A}_Q(E)$ to denote the $Q$-tuple given by the collection of $P_i$'s. If $E$ is endowed with a Euclidean metric, one can define a metric on $\mathcal{A}_Q(E)$ by 
$$
\dist\big(\sum_i [\![P_i]\!],\, \sum_i [\![S_i]\!]\big)=\min_{\sigma\in\mathcal{P}_Q}\sqrt{\sum_i |P_i-S_{\sigma(i)}|^2},
$$
where $\mathcal{P}_Q$ is the permutation group of $\{1,2,\cdots,Q\}$. If $T\in\mathcal{A}_Q(E)$, define $|T|=\dist(T,Q[\![0]\!])$.

A map from $X$ is called a $Q$-valued section of $\mathcal{V}$ if it maps every $x\in X$ to an element of $\mathcal{A}_Q(\mathcal{V}|_x)$. A $Q$-valued section is called continuous if it is continuous under local trivializations of $\mathcal{V}$.

\begin{Definition}
Let $U$ be a continuous $2$-valued section of $\mathcal{V}$. Then $U$ is called a $\mathbb{Z}/2$ harmonic spinor if the following conditions hold.
\begin{enumerate}
\item $U$ is not identically $2[\![0]\!]$.
\item Let $Z$ be the set of $U$ where $U=2[\![0]\!]$. For every $x\in X-Z$, there exists a neighborhood of $x$, such that on this neighborhood $U$ can be written as $U=[\![u]\!]+[\![-u]\!]$, where $u$ is a smooth section of $\mathcal{V}$ satisfying $D(u)=0$.
\item Near a point $x\in X-Z$, write $U$ as $[\![u]\!]+[\![-u]\!]$, then the function $|\nabla u|$ is a well defined smooth function on $X-Z$. The section $U$ satisfies
$$
\int_{X-Z} |\nabla u|^2<\infty.
$$
\end{enumerate}
\end{Definition}

This definition is equivalent to the definition of $\mathbb{Z}/2$ harmonic spinors given in \cite{taubesZ2}. 

For a point $x\in X$ and $r>0$, use $B_x(r)$ to denote the geodesic ball in $X$ with center $x$ and radius $r$. As in (1.5) of \cite{taubesZ2}, we make the following additional assumption on $U$.

\begin{Assumption}\label{assumption}
There exits a constant $\epsilon >0$ such that the following holds. For every $x\in X$ with $U(x)=2[\![0]\!]$, there exist constants $C,r_0>0$, depending on $x$, such that
$$
\int_{B_x(r)} |U(y)|^2 \,dy < C\cdot r^{4+\epsilon},\quad \text{for every }r\in(0,r_0). 
$$
\end{Assumption}

Assume $U$ is a $\mathbb{Z}/2$ harmonic spinor, and let $Z$ be the set of $U$ where $U=2[\![0]\!]$. Taubes \cite{taubesZ2} proved the following theorem.
\begin{Theorem}[Taubes \cite{taubesZ2}]\label{thm: Taubes}
If $U$ satisfies assumption \ref{assumption}, then the Hausdorff dimension of $Z$ is at most 2.
\end{Theorem}

This article improves theorem \ref{thm: Taubes} to the following result.
\begin{Theorem}\label{thm: main}
If $U$ satisfies assumption \ref{assumption}, then $Z$ is a $2$-rectifiable set. Moreover, for every compact subset $A\subset X$, there exist constants $C$ and $r_0$ depending on $A$ and $Z$, such that for every $r<r_0$,
$$\text{Vol }(\{x:\dist(x,A\cap Z)<r\})<C\cdot r^2.$$
\end{Theorem}

In other words, $Z$ is a $2$-rectifiable set with locally finite 2 dimensional Minkowski content. Since the Minkowski content controls the Hausdorff measure, theorem \ref{thm: main} implies that $Z$ has locally finite 2 dimensional Hausdorff measure.

Theorem \ref{thm: main} immediately implies that the zero locus of a $\mathbb{Z}/2$ harmonic spinor on a $3$-manifold is 1-rectifiable and has locally finite Minkowski content.

\section*{Acknowledgments}
This project originated from an idea of Clifford Taubes that one should be able to apply the techniques for Dir-minimizing $Q$-valued functions to the study of $\mathbb{Z}/2$ harmonic spinors. I would like to express my most sincere gratitude for his insightful guidance and encouragement. I also want to thank Thomas Walpuski and Aaron Naber for many helpful discussions and correspondences.

\section{$\mathbb{Z}/2$ harmonic spinors as Sobolev sections}
Almgren \cite{almgren2000almgren} developed a Sobolev theory for $Q$-valued functions on $\mathbb{R}^m$. For a quicker introduction, one can see for example \cite{Qrevisit}. For an open set $\Omega\subset\mathbb{R}^m$, the space $W^{1,2}(\Omega,\mathcal{A}_Q)$ is defined to be the space of $Q$ valued functions $T$ on $\Omega$, such that $|T|\in L^2(\Omega)$, and that $T$ has distributional derivatives which are also in $L^2(\Omega)$. The Sobolev theory extends to $Q$-valued sections of vector bundles without any difficulty. This section proves the following lemma. 

\begin{Lemma}\label{lem: Sobolev}
If $U$ is a $\mathbb{Z}/2$ harmonic spinor, then $U$ is in $W^{1,2}(X,\mathcal{A}_2)$. Moreover, $D(U)=0$ in the distributional sense.
\end{Lemma}
This lemma allows us to study the compactness properties of $\mathbb{Z}/2$ harmonic spinors by the Sobolev theory for $Q$-valued functions.

We start with the following definition.

\begin{Definition}
Let $T$ be a $Q$-valued section of $\mathcal{V}$. It is called a smooth $Q$-valued section, if for every $x\in X$, there exists a neighborhood of $x$ on which $T$ can be written as
$$
T=\sum_{i=1}^Q [\![f_i]\!],
$$
where $f_i$'s are smooth sections of $\mathcal{V}$.
\end{Definition}

If $T$ is a smooth $Q$-valued section and is locally written as $\sum_i [\![f_i]\!]$, then the function $\sum_i |f_i|^2+\sum_i |\nabla f_i|^2$ is well defined on $X$. In this case, the $W^{1,2}$ norm of $T$ is given by $(\int_X \sum_i |f_i|^2+\sum_i |\nabla f_i|^2)^{1/2}$.

\begin{proof}[Proof of lemma \ref{lem: Sobolev}]
The proof is essentially the same as lemma 2.4 of \cite{taubesZ2}.

Let $\chi$ be a smooth non-increasing function on $\mathbb{R}$, such that $\chi(t)=1$ when $t\le 1$, and $\chi(t)=0$ when $t\ge 2$. For $s>0$, let $\tau_s=\chi(\ln |U|/\ln s)$. Then $\tau_s(x)=0$ when $|U(x)|\le s^2$, and $\tau_s(x)=1$ when $|U(x)|\ge s$.

The section $\tau_s U$ is a $2$-valued smooth section of $\mathcal{V}$. Recall that on $X-Z$, the $\mathbb{Z}/2$ harmonic spinor $U$ can be locally written as $U=[\![u]\!]+[\![-u]\!]$. Although $u$ is only defined up to a sign, the functions $|u|$ and $| \tau_s \nabla u + \nabla \tau_s \cdot u |$ are well defined on $X-Z$. Thus the $W^{1,2}$ norm of $\tau_s  U$ is  given by
$$
\|\tau_s  U\|_{W^{1,2}}=\sqrt{2}\int_X (|\tau_s|^2 |u|^2 + | \tau_s \nabla u + \nabla \tau_s \cdot u |^2).$$
Notice that
$$
|\nabla \tau_s|\cdot |u|\le \frac{1}{|\ln s|} (\sup |\chi'|)\cdot |\nabla u|,
$$
hence its $L^2$ norm converges to zero as $s\to 0$.
Therefore, 
\begin{equation} \label{e: lim of L2 norm in weak approx}
\lim_{s\to 0} \|\tau_s  U\|_{W^{1,2}} = \sqrt{2} \int_{X-Z} (|u|^2+|\nabla u|^2).
\end{equation}
In particular, $\tau_s  U$ is bounded in $W^{1,2}$ as $s\to 0$, thus a subsequence of it weakly converges in $W^{1,2}$ to an element $U'\in W^{1,2}$. Since $\tau_s  U$ also uniformly converges to $U$, one must have $U'=U$. Therefore $U\in W^{1,2}$.

Since $D$ is a smooth first-order differential operator, $D(U)\in L^2_{loc}(X)$. By the definition of $\mathbb{Z}/2$ harmonic spinors, $D(U)=0$ on $X-Z$. By section 2.2.1 of \cite{Qrevisit}, the derivatives of $U$ are zero at the Lebesgue points of $Z$, hence $D(U)=0$ on those points. That proves $D(U)=0$ in the distributional sense.
\end{proof}

The argument of lemma 2.1 also shows that $U$ can be $W^{1,2}$ approximated by smooth sections. We write it as a separate lemma for later reference.

\begin{Lemma} \label{lem: smooth approx}
Let $U$ be a $\mathbb{Z}/2$ harmonic spinor. Then there exits a sequence of smooth sections $U_i$, such that $U_i=-U_i$, and 
$$\lim_{i\to\infty}U_i=U \text{ in } W^{1,2}.$$
\end{Lemma}

\begin{proof}
Since $|U|$ and $|\nabla U|$ are zero on the Lebesgue points of $Z$, one has
$$\| U\|_{W^{1,2}} = \int_{X-Z} (|U|^2+|\nabla U|^2)= \sqrt{2} \int_{X-Z} (|u|^2+|\nabla u|^2).$$

Define $\tau_s$ as in the proof of lemma \ref{lem: Sobolev}. It was proved previously that there is a sequence $s_i\to 0$, such that $\tau_{s_i}U$ converges weakly to $U$ in $W^{1,2}$. As a consequence,
$$
\liminf_{i\to\infty} \| \tau_{s_i} U\|_{W^{1,2}} \ge \| U\|_{W^{1,2}}
$$ 
On the other hand, by \eqref{e: lim of L2 norm in weak approx},
$$
\lim_{i\to\infty}\| \tau_{s_i} U\|_{W^{1,2}} = \sqrt{2} \int_{X-Z} (|u|^2+|\nabla u|^2)=\| U\|_{W^{1,2}}.
$$
Therefore $\tau_{s_i}U$ converges strongly to $U$ in $W^{1,2}$.
\end{proof}

\section{Frequency functions}
The frequency functions were first introduced by Amgren \cite{almgren1979dirichlet} to study the singular set of elliptic partial differential equations, and they were adapted by Taubes \cite{taubesZ2} to study the zero loci of $\mathbb{Z}/2$ harmonic spinors.
This section recalls some results about the frequency functions from \cite{taubesZ2}. 

Let $U$ be a $\mathbb{Z}/2$ harmonic spinor. On $X-Z$ the section $U$ can be locally written as $U=[\![u]\!]+[\![-u]\!]$. As before, we will use notations like $|u|$ and $|\nabla u|$ to denote the corresponding functions on $X-Z$ if they can be globally defined. The functions $|u|$ and $|\nabla u|$ extend to $X$ by defining them to be zero on $Z$.

The following $C^0$ estimate was established in \cite{taubesZ2}.

\begin{Lemma}[\cite{taubesZ2}, Lemma 2.3]\label{lem: C0 estimate}
Let $A\subset B$ be two open subsets of $X$, and assume the closure of $A$ is compact and contained in $B$. Then there exists a constant $K$, depending on $A$, $B$ and the norms of the curvatures of $X$ and $\mathcal{V}$, such that
$$
\sup_{x\in A} |u(x)|^2 \le K\int_B |u(x)|^2 \, dx.
$$
\end{Lemma}

Now introduce some notations. Fix a point $x_0\in X$. Take $R>0$ such that $B_{x_0}(500R)\subset X$ is complete, and that the injectivity radius of $X$ is greater than $1000R$ for every point in the ball $B_{x_0}(500R)$. 

Later on we will need to work on both the Euclidean space and the manifold $X$, so we need to differentiate the notations. We will use $B_x(r)$ to denote the geodesic ball on $X$ with center $x\in X$ and radius $r>0$. Use $\bar B_x(r)$ to denote the Euclidean ball with center $x$ in the Euclidean space and radius $r>0$. When the center is the origin, $\bar B(r)$ is also used to denote $\bar B_{0}(r)$. Use $d(x,y)$ to denote the distance function on $X$, and use $|x-y|$ to denote the distance function on $\mathbb{R}^4$.

For every $x\in B_{x_0}(500R)$, use the normal coordinate centered at $x$ to identify $B_{x}(500R)$ with the ball $\bar B(500R)\subset\mathbb{R}^4$. Let $g_x$ be the function of metric matrices on $\bar B(500R)$ corresponding to $B_x(500R)$. For each $z\in \bar B(500R)$, let $K_x(z),\kappa_x(z)$ be the largest and smallest eigenvalue of $g_x(z)$. Assume that $R$ is sufficiently small so that for every $x\in B_{x_0}(500R)$, $z\in \bar B(500R)$,
\begin{equation}\label{i: close to euclidean}
\big(\frac{11}{12}\big)^2\le \kappa_x(z) \le K_x(z) \le \big(\frac{12}{11}\big)^2
\end{equation}

In order to prove theorem \ref{thm: main}, one only needs to study the rectifiability and the Minkowski content of $Z\cap B_{x_0}(R/2)$. 

For $x\in B_{x_0}(500R)$, $r\in(0,500R]$, define the height function
$$
H(x,r)=\int_{\partial B_x(r)} |u|^2,
$$
then $H(x,r)$ is always positive \cite[Lemma 3.1]{taubesZ2}. Define
$$
D(x,r)=\int_{B_x(r)} |\nabla u|^2,
$$ 
and define the frequency function
$$
N(x,r)=\frac{rD(x,r)}{H(x,r)}.
$$
Section 3(a) of \cite{taubesZ2} proved the following monotonicity properties for $N$ and $H$:
\begin{Lemma}[\cite{taubesZ2}, (3.6) and Lemma 3.2]
The functions $N$ and $H$ are absolutely continuous with respect to $r$, and there exist constants $\kappa >0$ and $r_0>0$, depending only on the norms of curvatures of $X$ and $\mathcal{V}$ on $B_{x_0}(1000R)$, such that when $r\le r_0$,
\begin{align}
 \label{eqn: derivative of H}
\frac{\partial}{\partial r} H & \ge \frac{3}{r} H - \kappa rH, 
 \\
 \frac{\partial}{\partial r} N & \ge -\kappa r(1+N). \label{eqn: derivative of N}
 \\
  (\frac{N}{r} + \kappa r)\,\frac{H}{r^3}\ge & \frac{\partial}{\partial r}(\frac{H}{r^3}) \ge  (\frac{N}{r} - \kappa r)\,\frac{H}{r^3} \label{i: bound diviation of H by N}
\end{align}
\end{Lemma} 

By shrinking the size of $R$, we assume without loss of generality that $r_0= 500R$, hence inequalities \eqref{eqn: derivative of H}, \eqref{eqn: derivative of N}, and \eqref{i: bound diviation of H by N} hold for all $x\in B_{x_0}(500R)$ and $r\le 500R$.

Inequality \eqref{eqn: derivative of H} gives the following lemma
\begin{Lemma}[\cite{taubesZ2}, Lemma 3.1]\label{lem: mono of H}
There exists a constant $\kappa >0$, such that when $s<r<500R$,
$$
H(x,r)\ge \big(\frac{r}{s}\big)^3 \cdot e^{-\kappa (r^2-s^2)} \cdot H(x,s).
$$
\end{Lemma}
Inequality \eqref{eqn: derivative of N} gives

\begin{Lemma} \label{lem: mono of N}
There exists a constant $\kappa >0$, such that when $s<r<500R$,
$$
N(x,r)\ge e^{-\kappa(r^2-s^2)}N(x,s) - \kappa (r^2-s^2).
$$
\end{Lemma}

Since $N(x,500R)$ is continuous with respect to $x$, lemma \ref{lem: mono of N} implies that $N(x,r)$ is bounded for all $x\in B_{x_0}(500R)$, $r\le 500R$. Let $\Lambda$ be an upper bound for $N$. From now on $\Lambda$ will be treated as a constant. For the rest of this article, unless otherwise stated, $C$, $C_1$, $C_2$, $\cdots$ will denote positive constants that depend on $\Lambda$, $R$, and the norms of the curvatures of $X$ and $\mathcal{V}$, but independent of $U$. The values of $C$, $C_1$, $C_2$, $\cdots$ may be different in different appearances.

If $|g|\le C\cdot f$
for some constant $C$, we write $g=O(f)$.

Inequality \eqref{i: bound diviation of H by N} then implies that there exists a constant $C$ such that
\begin{equation} \label{i: bounded oscillation of H}
\Big| \frac{\partial}{\partial r}\big(\ln (\frac{H}{r^3})\big)\Big|= O(\frac{1}{r}).
\end{equation}
Inequality \eqref{eqn: derivative of N} implies that there exists $C>0$, such that whenever $r\ge s$,
$$
N(x,r) \ge N(x,s)-C(r^2-s^2).
$$

\section{Smoothed frequency functions}
We need to use a modified version of frequency functions.
Let $\phi$ be a non-increasing smooth function on $\mathbb{R}$ such that $\phi(t)=1$ when $t\le 3/4$, and $\phi(t)=0$ when $t\ge 1$. From now on $\phi$ will be fixed, hence the values of $\phi$ and its derivatives are considered as universal constants.
Following \cite{rec}, we define the smoothed frequency functions as follows.
\begin{Definition}
For $x\in X$, let $\nu_x$ be the gradient vector field of the distance function $d(x,\cdot)$. For $x\in B_{x_0}(500R)$, $r\le 500R$, introduce the following functions
\begin{align*}
D_{\phi}(x,r) &= \int |\nabla u(y)|^2 \phi\Big(\frac{d(x,y)}{r}\Big)\,dy,\\
H_{\phi}(x,r) &= -\int |u(y)|^2 d(x,y)^{-1} \phi'\Big(\frac{d(x,y)}{r}\Big) \, dy, \\
N_{\phi}(x,r) &= \frac{rD_{\phi}(x,r)}{H_{\phi}(x,r)},\\
E_{\phi}(x,r) &= -\int |\nabla_{\nu_x}u(y)|^2d(x,y)\phi'\Big(\frac{d(x,y)}{r}\Big) \, dy.
\end{align*}
\end{Definition}

Inequality \eqref{i: bounded oscillation of H} has the following useful corollary.
\begin{Lemma}
There exists a constant $C$ with the following property.
Let $r\in(0,32R]$. Assume $s_1\le 10r$, $s_2 \ge r/10$. Then for any two points $x$, $y$ with $d(x,y)\le r$, one has
$$
H_\phi(x,s_1)\le C(H_\phi(y,s_2)).
$$
\end{Lemma}

\begin{proof}
Since the constant $K$ in lemma 3.1 only depends on the norms of the
curvatures and the sets $A$, $B$, a rescaling argument gives
$$
|u(z)|^2 \le \frac{C_1}{r^4}\int_{B_{z}(r)} |u|^2,\quad \forall B_z(r)\subset B_{x_0}(500R).
$$
Therefore for every $z\in \partial B_x(s_1)$, 
$$
|u(z)|^2 \le \frac{C_2}{r^4} \int_{B_y(12r)} |u|^2.
$$
On the other hand, inequality \eqref{i: bounded oscillation of H} and lemma \ref{lem: mono of H} 
 gives
$$
 \frac{1}{r^4} \int_{B_y(12r)} |u|^2 \le \frac{C_3}{r^3} H(y,s_2).
$$
Therefore
$$
H(x,s_1)=O(H(y,s_2)).
$$
Apply \eqref{i: bounded oscillation of H} again, one obtains  
\begin{align*}
H(y,s_2) &= O(H_\phi(y,s_2)), \\
H_\phi(x,s_1) &= O(H(x,s_1)), 
\end{align*}
hence the lemma is proved.
\end{proof}

\begin{Lemma} \label{lem: bound |u||nabla u|}
For $x\in B_{x_0}(32R)$, $r\le 32R$, one has
$$
\int_{B_x(r)} |u(y)|^2 dy = O(rH_\phi(x,r)),
$$
$$
\int_{B_x(r)} |u(y)||\nabla u(y)| dy = O(H_\phi(x,r)),
$$
$$
\int_{B_x(r)} |\nabla u(y)|^2 dy = O(\frac{1}{r}H_\phi(x,r)).
$$
\end{Lemma}
\begin{proof}
The first equation follows from inequality \eqref{i: bounded oscillation of H} and lemma \ref{lem: mono of H}. 
For the third,
\begin{align*}
\int_{B_x(r)} |\nabla u(y)|^2 dy &\le D_\phi(x,2r) \\ 
&= \frac{1}{2r} N_\phi(x,2r) H_\phi(x,2r) \\
&= O(\frac{1}{r}H_\phi(x,r)).
\end{align*}
The second equation then follows from Cauchy's inequality.
\end{proof}

The main result of this section is the following proposition.
\begin{Proposition} \label{prop: int by parts}
The functions $D_\phi$, $H_\phi$, $N_\phi$, and $E_\phi$ are smooth in both variables. Assume $x\in B_{x_0}(32R)$, $r\le 32R$, and $v\in T_x(X)$. Consider the normal coordinate centered at $x$ with radius $r$, extend the vector $v$ to a vector field on $B_x(r)$ by requiring that the coordinate functions of $v$ are constants. Then the following equations hold
\begin{align}
D_\phi(x,r) &= -\frac{1}{r}\int \phi'\Big(\frac{d(x,y)}{r}\Big) \nabla_{\nu_x}u(y)\cdot u(y)\,dy  + O(rH_\phi(x,r)),\label{eqn: int by parts 1}\\
\partial_r D_\phi(x,r) &= \frac{2}{r} D_\phi(x,r) + \frac{2}{r^2} E_\phi(x,r) +O(H_\phi(x,r)), \label{eqn: int by parts 2}\\
\partial_v D_\phi(x,r) &=   -\frac{2}{r}\int \phi'\Big(\frac{d(x,y)}{r}\Big) \nabla_{\nu_x}u(y)\cdot \nabla_v u(y)\,dy + O(H_\phi(x,r)), \label{eqn: int by parts 3}\\
\partial_r H_\phi(x,r) &= \frac{3}{r}H_\phi(x,r) + 2 D_\phi(x,r) + O(rH_\phi(x,r)),
\label{eqn: int by parts 4} \\
\partial_v H_\phi(x,r) &= -2\int u(y)\cdot\nabla_v u(y) \, d(x,y)^{-1} \phi'\Big(\frac{d(x,y)}{r}\Big) \, dy + O(rH_\phi(x,r)). \label{eqn: int by parts 5}
\end{align}
\end{Proposition}

The smoothness of the functions follows from the fact that $\phi$ is smooth and $|u|$, $|\nabla u|$ are both in $L^2$.

\begin{proof}[Proof of \eqref{eqn: int by parts 1}]
It was proved in \cite[Section 2(c)]{taubesZ2} that
\begin{equation} \label{eqn: der of H in r}
\int_{\partial B_x(s)} \nabla_{\nu_x}u(y)\cdot u(y)\, dy = \int_{B_x(s)}|\nabla u(y)|^2\,dy + \int_{B_x(s)}\langle u(y),\mathcal{R} u(y) \rangle \,dy,
\end{equation}
where $\mathcal{R}$ is a bounded curvature term from the Weitzenb\"ock formula.

Therefore, by lemma \ref{lem: bound |u||nabla u|},
\begin{align*}
D_\phi(x,r) &= -\frac{1}{r}\int_0^r \phi'\Big(\frac{s}{r}\Big) \int_{B_x(s)}|\nabla u(y)|^2\,dy\,ds\\
&= -\frac{1}{r}\int \phi'\Big(\frac{d(x,y)}{r}\Big) \nabla_{\nu_x}u(y)\cdot u(y)\,dy +\frac{1}{r}\int_0^r \phi'\Big(\frac{s}{r}\Big)\int_{B_x(s)} \langle u,\mathcal{R} u \rangle \,dy \,ds\\
&= -\frac{1}{r}\int \phi'\Big(\frac{d(x,y)}{r}\Big) \nabla_{\nu_x}u(y)\cdot u(y)\,dy + O(rH_\phi(x,r)).
\end{align*}
\end{proof}

\begin{proof}[Proof of \eqref{eqn: int by parts 2}]
\begin{align}
\partial_r D_\phi(x,r) &= -\frac{1}{r^2}\int |\nabla u(y)|^2 \phi'\Big(\frac{d(x,y)}{r}\Big)\cdot d(x,y) \, dy \nonumber
\\
&= -\frac{1}{r^2} \int_0^r \phi'\Big(\frac{s}{r}\Big)\cdot s \int_{\partial B_x(s)} |\nabla u(y)|^2 \, dy\,ds \label{eqn: int by parts 2 step 1}
\end{align}
It was proved in \cite[Section 2(d)]{taubesZ2} that
\begin{multline*}
\int_{\partial B_x(s)} |\nabla u(y)|^2 \,dy= 
2 \int_{\partial B_x(s)} |\nabla_{\nu_x} u(y)|^2 \,dy
+ \frac{2}{s} \int_{B_x(s)} |\nabla u(y)|^2\,dy
\\
+ \frac{2}{s} \int_{B_x(s)} \langle u(y), \mathcal{R} u(y) \rangle \,dy 
-
\int_{\partial B_x(s)} \langle \mathcal{R}_1 u(y),\nabla u(y)\rangle\, dy
+\int_{\partial B_x(s)} \langle u(y),\mathcal{R}_2 u(y) \rangle\,dy,
\end{multline*}
where $\mathcal{R}$, $\mathcal{R}_1$, $\mathcal{R}_2$ are smooth tensors, $\mathcal{R}$ and $\mathcal{R}_2$ are bounded, the norm of $\mathcal{R}_1$ is bounded by $C_1\cdot r$.

Notice that
$$
-\int_0^r \phi'\Big(\frac{s}{r}\Big)\cdot s \int_{\partial B_x(s)} |\nabla_{\nu_x} u(y)|^2 \,dy\,ds = E_\phi(x,r),
$$
$$
-\frac{1}{r}\int_0^r \phi'\Big(\frac{s}{r}\Big) \int_{B_x(s)} |\nabla u(y)|^2\,dy\,ds = D_\phi(x,r).
$$
Plug into equation \eqref{eqn: int by parts 2 step 1}, we have
\begin{multline*}
\partial_r D_\phi (x,r) = \frac{2}{r} D_\phi(x,r)+\frac{2}{r^2} E_\phi(x,r)
-\frac{1}{r^2} \int_0^r \phi'\Big(\frac{s}{r}\Big)\cdot s \cdot\Big[\frac{2}{s} \int_{B_x(s)} \langle u(y), \mathcal{R} u(y) \rangle \,dy 
\\
-\int_{\partial B_x(s)} \langle \mathcal{R}_1 u(y),\nabla u(y)\rangle \, dy
+\int_{\partial B_x(s)} \langle u(y),\mathcal{R}_2 u(y) \rangle\,dy\Big] \,ds.
\end{multline*}
Lemma \ref{lem: bound |u||nabla u|} implies
\begin{samepage}
\begin{multline*}
-\frac{1}{r^2} \int_0^r \phi'\Big(\frac{s}{r}\Big)\cdot s \cdot\Big[\frac{2}{s} \int_{B_x(s)} \langle u(y), \mathcal{R} u(y) \rangle \,dy 
+\int_{\partial B_x(s)} \langle u(y),\mathcal{R}_2 u(y) \rangle\,dy\Big] \,ds \\ = O(H_\phi(x,r)).
\end{multline*}
\end{samepage}
On the other hand,
\begin{align*}
& \Big|-\frac{1}{r^2} \int_0^r \phi'\Big(\frac{s}{r}\Big)\cdot s \cdot\Big[-\int_{\partial B_x(s)} \langle\mathcal{R}_1 u(y),\nabla u(y)\rangle \, dy\Big]\,ds \Big| 
\\
\le & C_2\cdot \int_0^r  \Big|\phi'\Big(\frac{s}{r}\Big) \Big|\int_{\partial B_x(s)} |u(y)||\nabla u(y)| \, dy\,ds 
\\
\le &
C_3\int_{B_x(r)} |u(y)||\nabla u(y)| dy 
=O(H_\phi(x,r)).
\end{align*}
Hence the result is proved.
\end{proof}

\begin{proof}[Proof of \eqref{eqn: int by parts 3}]
For a function $G(x,y)$ defined on $X\times X$ and a vector field $w$, use 
$\frac{\partial x}{\partial w}G$ to denote the directional derivative of $G$ with respect to $x$, use $\frac{\partial y}{\partial w}G$ to denote the directional derivative with respect to $y$.

The first variation formula of geodesic lengths gives
$$
\frac{\partial x}{\partial v}d(x,y) + \frac{\partial y}{\partial v}d(x,y) = O(d(x,y)^2).
$$
We have
\begin{align}
\frac{\partial x}{\partial v} D_\phi(x,r) &= \frac{1}{r} \int |\nabla u(y)|^2 \phi'\Big(\frac{d(x,y)}{r}\Big) \cdot \frac{\partial x}{\partial v}d(x,y)\,dy \nonumber \\
&= -\frac{1}{r} \int |\nabla u(y)|^2 \phi'\Big(\frac{d(x,y)}{r}\Big) \cdot \frac{\partial y}{\partial v}d(x,y)\,dy + O(r)\int_{B_x(r)} |\nabla u(y)|^2\nonumber\\
&= - \int |\nabla u(y)|^2  \cdot \frac{\partial y}{\partial v}\phi\Big(\frac{d(x,y)}{r}\Big)\,dy + O(H_\phi(x,r)). \label{eqn: int by parts 3 step 1}
\end{align}
One needs to establish the following lemma.

\begin{Lemma} \label{lem: int by parts 3 step 1}
Let $F$ be the curvature of $\mathcal{V}$, and $\{e_i\}$ be an orthonormal basis of $TX$. Let $\varphi$ be a smooth function with $\supp\,\varphi\subset B_x(r)$. Then
\begin{align*}
\quad &\int |\nabla u|^2 \partial_v \varphi \\
&= 
 2 \int\langle d\varphi\otimes\nabla_v u , \nabla u \rangle -2\int \sum_i \varphi \langle F(v,e_i) u, \nabla_{e_i}u \rangle -2\int \sum_i \varphi \langle \nabla_{[v,e_i]} u, \nabla_{e_i} u\rangle
\\
&\quad
 - \int |\nabla u|^2\varphi \,\text{div}(v)  +2\int\sum_i \varphi \langle \nabla_v u,\nabla_{\nabla_{e_i}e_i}u\rangle \\ 
&\quad
+2
\int \sum_i \varphi \langle \nabla_v u,\nabla_{e_i} u\rangle\,\text{div} (e_i)
+2\int  \varphi \langle \nabla_v u, \mathcal{R}_0u \rangle,
\end{align*}
where $\mathcal{R}_0$ is the curvature term in the Weitzenb\"ock formula.
\end{Lemma}

\begin{proof}[Proof of lemma \ref{lem: int by parts 3 step 1}]
By lemma \ref{lem: smooth approx}, there exists a sequence of smooth 2-valued section $U_i$, such that $U_i=-U_i$ and $U_i\to U$ in $W^{1,2}$. By partitions of unity, integration by parts works for $U_i$. For any $U_i$, locally write it as $[\![w]\!]+[\![-w]\!]$ where $w$ is a smooth section of $\mathcal{V}$, then
\begin{align*}
&\quad \int |\nabla w|^2 {\partial_v}\varphi\\
&=  -\int \sum_i \varphi \nabla_v\langle \nabla_{e_i} w, \nabla_{e_i} w  \rangle - \int |\nabla w|^2\varphi \,\text{div}(v)\\
&= -2 \int \sum_i \varphi \langle \nabla_{e_i}\nabla_v w, \nabla_{e_i} w  \rangle  -2\int \sum_i \varphi \langle F(v,e_i) w, \nabla_{e_i}w \rangle \\
&\quad -2\int \sum_i \varphi \langle \nabla_{[v,e_i]} w, \nabla_{e_i} w\rangle - \int |\nabla w|^2\varphi \,\text{div}(v)
\end{align*}
Here $F$ denotes the curvature of $\mathcal{V}$. For the first term in the formula above,
\begin{align*}
&\quad \int \sum_i \varphi \langle \nabla_{e_i}\nabla_v w, \nabla_{e_i} w  \rangle 
\\
&= - \int\sum_i (\nabla_{e_i}\varphi) \langle \nabla_v w, \nabla_{e_i} w  \rangle -\int  \sum_i \varphi \langle \nabla_v w, \nabla_{e_i}\nabla_{e_i} w  \rangle 
\\
& \quad  -
\int \sum_i \varphi \langle \nabla_v w,\nabla_{e_i} w\rangle\,\text{div} (e_i)
\\
&= - \int\sum_i (\nabla_{e_i}\varphi) \langle \nabla_v w, \nabla_{e_i} w  \rangle +\int  \sum_i \varphi \langle \nabla_v w, \nabla^\dagger\nabla w  \rangle 
\\
& \quad  -\int\sum_i \varphi \langle \nabla_v w,\nabla_{\nabla_{e_i}e_i}w\rangle-
\int \sum_i \varphi \langle \nabla_v w,\nabla_{e_i} w\rangle\,\text{div} (e_i)
\end{align*}
For the second term in the formula above,  let $\mathcal{R}_0$ be the curvature term in the Weitzenb\"ock formula, then
\begin{align*}
&\quad \int  \sum_i \varphi \langle \nabla_v w, \nabla^\dagger\nabla w  \rangle = \int   \langle \varphi \nabla_v w, D^2 w-\mathcal{R}_0w \rangle 
\\
&=   -\int  \varphi \langle \nabla_v w, \mathcal{R}_0w \rangle 
+  \int \langle \rho(\nabla\varphi)\nabla_v w , Dw \rangle
- \int \langle \varphi \langle [\nabla_v,D]w, Dw\rangle  
 +  \int \varphi\langle \nabla_v(Dw), Dw \rangle 
\\
&=  -\int  \varphi \langle \nabla_v w, \mathcal{R}_0w \rangle 
+  \int \langle \rho(\nabla\varphi)\nabla_v w , Dw \rangle 
- \int \langle \varphi \langle [\nabla_v,D]w, Dw\rangle 
\\
&\quad 
 -\frac{1}{2}  \int \partial_v\varphi |Dw|^2 
 -\frac{1}{2} \int \varphi |Dw|^2 \,\text{div}(v)
\end{align*}
Therefore
\begin{align*}
&\quad \int |\nabla w|^2 {\partial_v}\varphi
\\
&= 
 -2\int \sum_i \varphi \langle F(v,e_i) w, \nabla_{e_i}w \rangle -2\int \sum_i \varphi \langle \nabla_{[v,e_i]} w, \nabla_{e_i} w\rangle- \int |\nabla w|^2\varphi \,\text{div}(v)\quad
\\
&\quad
+2 \int\sum_i (\nabla_{e_i}\varphi) \langle \nabla_v w, \nabla_{e_i} w  \rangle   +2\int\sum_i \varphi \langle \nabla_v w,\nabla_{\nabla_{e_i}e_i}w\rangle +2
\int \sum_i \varphi \langle \nabla_v w,\nabla_{e_i} w\rangle\,\text{div} (e_i)
\\
&\quad
+2\int  \varphi \langle \nabla_v w, \mathcal{R}_0w \rangle 
-2  \int \langle \rho(\nabla\varphi)\nabla_v w , Dw \rangle 
+2 \int \langle \varphi \langle [\nabla_v,D]w, Dw\rangle 
\\
&\quad 
 +  \int \partial_v\varphi |Dw|^2 
 - \int \varphi |Dw|^2 \,\text{div}(v)
\end{align*}
Take limit $U_i\to U$, one has
\begin{align*}
&\quad \int |\nabla u|^2 {\partial_v}\varphi 
\\
&= 
 -2\int \sum_i \varphi \langle F(v,e_i) u, \nabla_{e_i}u \rangle -2\int \sum_i \varphi \langle \nabla_{[v,e_i]} u, \nabla_{e_i} u\rangle- \int |\nabla u|^2\varphi \,\text{div}(v)\quad
\\
&\quad
+2 \int\sum_i (\nabla_{e_i}\varphi) \langle \nabla_v u, \nabla_{e_i} u  \rangle   +2\int\sum_i \varphi \langle \nabla_v u,\nabla_{\nabla_{e_i}e_i}u\rangle +2
\int \sum_i \varphi \langle \nabla_v u,\nabla_{e_i} u\rangle\,\text{div} (e_i)
\\
&\quad
+2\int  \varphi \langle \nabla_v u, \mathcal{R}_0u \rangle 
-2  \int \langle \rho(\nabla\varphi)\nabla_v u , Du \rangle 
+2 \int \langle \varphi \langle [\nabla_v,D]u, Du\rangle 
\\
&\quad 
 +  \int \partial_v\varphi |Du|^2 
 - \int \varphi |Du|^2 \,\text{div}(v) 
\\
&= 
 -2\int \sum_i \varphi \langle F(v,e_i) u, \nabla_{e_i}u \rangle -2\int \sum_i \varphi \langle \nabla_{[v,e_i]} u, \nabla_{e_i} u\rangle- \int |\nabla u|^2\varphi \,\text{div}(v)\quad
\\
&\quad
+2 \int\sum_i (\nabla_{e_i}\varphi) \langle \nabla_v u, \nabla_{e_i} u  \rangle   +2\int\sum_i \varphi \langle \nabla_v u,\nabla_{\nabla_{e_i}e_i}u\rangle \\
&\quad
+2
\int \sum_i \varphi \langle \nabla_v u,\nabla_{e_i} u\rangle\,\text{div} (e_i)
+2\int  \varphi \langle \nabla_v u, \mathcal{R}_0u \rangle 
\end{align*}
Notice that 
$$
\sum_i (\nabla_{e_i}\varphi) \langle \nabla_v u, \nabla_{e_i} u  \rangle= \langle d\varphi\otimes\nabla_v u , \nabla u \rangle,
$$
therefore the lemma is proved.
\end{proof}

Back to the proof of equation \eqref{eqn: int by parts 3}. Take $\varphi(y)=\phi({d(x,y)}/{r})$. By Lemma \ref{lem: bound |u||nabla u|},
$$
-2\int \sum_i \varphi \langle F(v,e_i) u, \nabla_{e_i}u \rangle + 2 \int  \varphi \langle \nabla_v u, \mathcal{R}_0u \rangle = O(H_\phi(x,r)).
$$
On the other hand, $|\text{div}(v)|=O(r)$, and one can choose $\{e_i\}$ such that $|[v,e_i]|=O(r)$, $|\text{div}(e_i)|=O(r)$, and $|\nabla_{e_i}e_i|=O(r)$. Thus by lemma \ref{lem: bound |u||nabla u|},
\begin{multline*}\quad
 -2\int \sum_i \varphi \langle \nabla_{[v,e_i]} u, \nabla_{e_i} u\rangle - \int |\nabla u|^2\varphi \,\text{div}(v)  +2\int\sum_i \varphi \langle \nabla_v u,\nabla_{\nabla_{e_i}e_i}u\rangle \\
\quad
+2
\int \sum_i \varphi \langle \nabla_v u,\nabla_{e_i} u\rangle\,\text{div} (e_i) = O(H_\phi(x,r)).
\end{multline*}

Equation \eqref{eqn: int by parts 3} then follows immediately from equation \eqref{eqn: int by parts 3 step 1} and lemma \ref{lem: int by parts 3 step 1}.
\end{proof}

\begin{proof}[Proof of \eqref{eqn: int by parts 4}]
By \cite[Equation (2.11)]{taubesZ2},
\begin{equation}\label{eqn: der of H int by parts}
\partial_s H(x,s)=\frac{3}{s}H(x,s) + 2D(x,s) + \int_{B_x(s)} \langle u, \mathcal{R} u \rangle + \int_{\partial B_x(s)} \mathfrak{t}|u|^2,
\end{equation}
where $\mathcal{R}$ is a curvature term from the Weitzenb\"ock formula, and $\mathfrak{t}$ comes from the mean curvature of $\partial B_x(s)$. The function $\mathfrak{t}$ satisfies $|\mathfrak{t}(y)|=O(d(x,y))$.
Notice that
$$
H_\phi(x,r)
=
\int_0^r -\phi'(s/r)\cdot \frac{1}{s} \cdot H(s) \, ds
=
\int_0^1 -\phi'(\lambda) \frac{1}{\lambda}\cdot H(\lambda r)\,d\lambda.
$$
Therefore
\begin{align*}
&\partial_r H_\phi(x,r) 
\\ =&
\int_0^1 -\phi'(\lambda) \cdot (\partial_r H)(\lambda r)\,d\lambda
\\ =& 
\int_0^1 -\phi'(\lambda) \Big[\frac{3}{\lambda r}H(x,\lambda r) + 2D(x,\lambda r) + \int_{B_x(\lambda r)} \langle u, \mathcal{R} u \rangle + \int_{\partial B_x(\lambda r)} \mathfrak{t}|u|^2 \Big]\,d\lambda
\\ =&
-\frac{1}{r}\int_0^r \phi'(s/r) \Big[\frac{3}{s}H(x,s) + 2D(x,s) + \int_{B_x(s)} \langle u, \mathcal{R} u \rangle + \int_{\partial B_x(s)} \mathfrak{t}|u|^2 \Big]\,ds
\\ 
=&
\frac{3}{r}H_\phi(x,r) + 2 D_\phi(x,r) -\frac{1}{r}\int_0^r \phi'(s/r) \Big[\int_{B_x(s)} \langle u, \mathcal{R} u \rangle + \int_{\partial B_x(s)} \mathfrak{t}|u|^2 \Big]\,ds
\\
=&
\frac{3}{r}H_\phi(x,r) + 2 D_\phi(x,r) + O(rH_\phi(x,r)).
\end{align*}
\end{proof}

\begin{proof}[Proof of \eqref{eqn: int by parts 5}]
As in the proof of \eqref{eqn: int by parts 3}, for a function $G(x,y)$, use $\frac{\partial x}{\partial v}G$ to denote the directional derivative of $G$ with respect to $x$, and use $\frac{\partial y}{\partial v}G$ to denote the directional derivative with respect to $y$. Recall that we have
$$\frac{\partial x}{\partial v}d(x,y) + \frac{\partial y}{\partial v}d(x,y) = O(d(x,y)^2),$$
therefore
$$
(\frac{\partial x}{\partial v}+\frac{\partial y}{\partial v})\Big[d(x,y)^{-1} \phi'\Big(\frac{d(x,y)}{r}\Big)\Big] = O(1).
$$
We have
\begin{align*}
&
\partial_v H(x,r) 
\\ =& 
-\int |u(y)|^2 \frac{\partial x}{\partial v}\Big[d(x,y)^{-1} \phi'\Big(\frac{d(x,y)}{r}\Big) \Big]dy
\\ =&
\int |u(y)|^2 \frac{\partial y}{\partial v}\Big[d(x,y)^{-1} \phi'\Big(\frac{d(x,y)}{r}\Big) \Big]dy + O(\int_{B_x(r)}|u|^2)
\\ =&
-\int \frac{\partial }{\partial v} |u(y)|^2 d(x,y)^{-1} \phi'\Big(\frac{d(x,y)}{r}\Big) dy 
\\ \quad &
- \int |u(y)|^2d(x,y)^{-1} \phi'\Big(\frac{d(x,y)}{r}\Big)\,\text{div}(v) dy+ O(rH_\phi(x,r))
\\ =&
 -2\int u(y)\cdot\nabla_v u(y) \, d(x,y)^{-1} \phi'\Big(\frac{d(x,y)}{r}\Big) \, dy + O(rH_\phi(x,r))
\end{align*}
The last equality follows from $|\text{div}(v)|=O(r)$ and $\int_{B_x(r)}|u|^2=O(rH_\phi(x,r))$.
\end{proof}

\begin{remark} \label{r: flat case}
When both $X$ and $\mathcal{V}$ are flat, all the curvature terms in the computations above are zero. Therefore, proposition \ref{prop: int by parts} becomes
\begin{align*}
D_\phi(x,r) &= -\frac{1}{r}\int \phi'\Big(\frac{d(x,y)}{r}\Big) \nabla_{\nu_x}u(y)\cdot u(y)\,dy,\\
\partial_r D_\phi(x,r) &= \frac{2}{r} D_\phi(x,r) + \frac{2}{r^2} E_\phi(x,r)\\
\partial_v D_\phi(x,r) &=   -\frac{2}{r}\int \phi'\Big(\frac{d(x,y)}{r}\Big) \nabla_{\nu_x}u(y)\cdot \nabla_v u(y)\,dy\\
\partial_r H_\phi(x,r) &= \frac{3}{r}H_\phi(x,r) + 2 D_\phi(x,r)\\
\partial_v H_\phi(x,r) &= -2\int u(y)\cdot\nabla_v u(y) \, d(x,y)^{-1} \phi'\Big(\frac{d(x,y)}{r}\Big) \, dy 
\end{align*}
\end{remark}

\begin{Corollary}\label{cor: der of N}
Let $\eta_x(y)=d(x,y)\cdot\nu_x(y)$. Under the assumptions of proposition \ref{prop: int by parts}, one has
\begin{multline} \label{eqn: v der of N}
\partial_v N_\phi(x,r)=\frac{2}{H_\phi(x,r)} \int -\frac{1}{d(x,y)}\phi'\Big(\frac{d(x,y)}{r}\Big)\cdot
\\
(\nabla_{\eta_x}u(y)-N_\phi(x,r) u(y))\cdot \nabla_v u(y) \,dy  + O(r).
\end{multline}
\begin{multline} \label{eqn: r der of N}
\partial_r N_\phi(x,r)=\frac{2}{rH_\phi(x,r)}\int -\phi'\Big(\frac{d(x,y)}{r}\Big)\cdot
\\
 d(x,y)^{-1} |\nabla_{\eta_x} u(y) - N_\phi(x,r) u(y) |^2 \, dy + O(r),
\end{multline}
As a consequence, there exists a constant $C$, such that $\Big(N_\phi(x,r)+Cr^2\Big)$ is increasing in $r$.
\end{Corollary}

\begin{proof}
The first equation follows immediately from proposition \ref{prop: int by parts} by combining equations \eqref{eqn: int by parts 3} and \eqref{eqn: int by parts 5}.
For the first one, lemma \ref{prop: int by parts} gives
$$
\partial_r N_\phi(x,r) = \frac{2}{rH_\phi(x,r)}\Big(E_\phi(x,r)-\frac{r^2D_\phi(x,r)^2}{H_\phi(x,r)}\Big)+O(r),
$$
and we have
\begin{align*}
& E_\phi(x,r)-\frac{r^2D_\phi(x,r)^2}{H_\phi(x,r)}
\\ = & 
E_\phi(x,r)-2rD_\phi(x,r) N_\phi(x,r) + N_\phi(x,r)^2 H_\phi(x,r) 
\\ = &
\int -\phi'\Big(\frac{d(x,y)}{r}\Big) d(x,y)^{-1} |\nabla_{\eta_x} u(y) - N_\phi(x,r) u(y) |^2 \, dy +O(r^2H_\phi(x,r))
\end{align*}
Hence the second equation is verified.

\end{proof}

\section{Compactness}
This section proves a compactness result for $\mathbb{Z}/2$ harmonic spinors.

Consider the ball $\Omega=\bar B(5)\subset \mathbb{R}^4$ centered at the origin. Let $\mathcal{V}$ be a fixed trivial vector bundle on $\Omega$. Assume $g_n$ is a sequence of Riemannian metrics on $\Omega$, $A_n$ is a sequence of connenction forms on $\mathcal{V}$, and $\rho_n$ is a sequence of Clifford bundle structures of $\mathcal{V}$. Assume that $(g_n,A_n,\rho_n)$ are compatible, and assume that
$(g_n,A_n,\rho_n)$ converge to $(g,A,\rho)$ in $C^\infty$. Assume $g$ is the Euclidean metric on $\bar B(5)$. Then for sufficiently large $n$, the injectivity radius at each point in $B(2)$ is at least $2.5$. Without loss of generality, assume that this property holds for every $n$.

Fix $\epsilon, \Lambda>0$. For every $n$, assume $U_n$ is a 2-valued section of $\mathcal{V}$ defined on $\bar B(5)$, with the following properties:
\begin{enumerate}
\item The section $U_n$ is a $\mathbb{Z}/2$ harmonic spinor on $\bar B(5)$ with respect to $(g_n,A_n,\rho_n)$.
\item $U_n$ satisfies assumption \ref{assumption} with respect to $\epsilon$.
\item Let $N^{(n)}_\phi$ be the smoothed frequency function for the extended $U_n$. Then whenever $N_\phi(x,r)$ is defined,
$$
N^{(n)}_\phi(x,r)\le \Lambda.
$$
\item Let $H^{(n)}_\phi$ be the smoothed height function of $U_n$, then $H^{(n)}_\phi(0,1)= 1$.
\end{enumerate}

The main result of this section is the following proposition.

\begin{Proposition} \label{p: compact}
Let $U_n$ be given as above. Then there exits a subsequence of $\{U_n\}$, such that the sequence converges strongly in $W^{1,2}(\bar B(2))$ to a section $U$. The section $U$ is a $\mathbb{Z}/2$ harmonic spinor on $\bar B(2)$ with respect to $(g,A,\rho)$, and $U$ satisfies assumption \ref{assumption} for a possibly smaller value of $\epsilon$. Moreover, $U_n$ converges to $U$ uniformly on $\bar B(2)$.
\end{Proposition}

\begin{proof}
Fix a trivialization of $\mathcal{V}$, and fix $s\in(0,0.5)$.
The bound on $N_\phi^{(n)}$ and the assumption that $H^{(n)}_\phi(0,1)= 1$ implies that $\|U\|_{L^2(\bar B(2+s))} \le C_1$ for some constant $C_1$. The upper bound on $N_\phi$ then implies
$\|\nabla_{A_n}U\|_{L^2(\bar B(2+s/2))} \le C_2$.
Since $A_n\to A$ in $C^\infty$, this implies that $U_n$ is bounded in $W^{1,2}(\bar B(2+s/2))$. Therefore, there is a subsequence of $\{U_n\}$ which converges weakly in $W^{1,2}(\bar B(2+s/2))$ and converges strongly in $L^2(\bar B(2+s/2))$. To avoid complicated notations, the subsequence is still denoted by $\{U_n\}$. Denote the limit of $\{U_n\}$ on $\bar B(2+s/2)$ by $U$. Let $H^{(n)}_\phi$, $D^{(n)}_\phi$, $N^{(n)}_\phi$ be the smoothed frequency functions for $U_n$, let $H_\phi$, $D_\phi$, $N_\phi$ be the corresponding functions for $U$. Since $U_n\to U$ strongly in $L^2$, one has $H_\phi(0,1)=1$, thus $U$ is not identically $2[\![0]\!]$.

By \cite[Section 3(e)]{taubesZ2}, there exists constants $K>0$ and $\alpha\in(0,1)$, depending on $\epsilon$, $\Lambda$, $R$ and the $C^1$ norms of the curvatures of $\{g_n\}$ and $A_n$, such that
$$
\|U_n\|_{C^\alpha(\bar B(2+s/2))} \le K.
$$
By the Arzela-Ascoli theorem, there exists a further subsequence of $\{U_n\}$ which converges uniformly to $U$ on $\bar B(2+s/2)$. Still denote this subsequence by $\{U_n\}$. Since solutions to the Dirac equation are closed under $C^0$ limits, $U$ is a $\mathbb{Z}/2$ harmonic spinor. $U$ is also H\"older continuous, so it satisfies assumption \ref{assumption}. 

Locally write $U_n$ as $[\![u_n]\!]+[\![-u_n]\!]$, and write $U$ as $[\![u]\!]+[\![-u]\!]$. The weak convergence of $U_n$ to $U$ implies
$$
\liminf_{n\to\infty} \int_{\bar B(2)}|\nabla_{A_n}u_n|^2 \ge \int_{\bar B(2)}|\nabla_{A}u|^2.
$$
We want to prove that 
$$
\lim_{n\to\infty} \int_{\bar B(2)}|\nabla_{A_n}u_n|^2 = \int_{\bar B(2)}|\nabla_{A}u|^2.
$$
Assume the contrary, then there exists a subsequence of $n$ such that 
$$
\int_{\bar B(2)}|\nabla_{A_n}u_n|^2 \ge \int_{\bar B(2)}|\nabla_{A}u|^2+\delta
$$
for some $\delta>0$.
Since $\int_{\bar B(r)}|\nabla_{A}u|^2$ is continuous in $r$, and $\int_{\bar B(r)}|\nabla_{A_n}u_n|^2$ is non-decreasing in $r$ for every $n$, there exists $r\in(2,2+s/2)$ and $\sigma\in(1,(2+s/2)/r)$, such that for every $t\in[2,r]$,
\begin{equation}\label{i: compare of int nabla u}
\int_{\bar B(t)}|\nabla_{A_n}u_n|^2 \ge \int_{\bar B(\sigma t)}|\nabla_{A}u|^2+\delta/2
\end{equation}
Use $B_n(t)$ to denote the geodesic ball of center $0$ and radius $t$ with metric $g_n$. Since $g_n\to g$, we have $\bar B(t)\subset B_n(\sigma t)$ for sufficiently large $n$. Equation \eqref{i: compare of int nabla u} then gives
\begin{equation}\label{i: compare of int nabla u in geodesic ball}
\int_{B_n(\sigma t)}|\nabla_{A_n}u_n|^2 \ge \int_{\bar B( \sigma t)}|\nabla_{A}u|^2+\delta/2, \quad \text{for }t\in [2,r]
\end{equation}
when $n$ is sufficiently large.

By equation \eqref{eqn: der of H int by parts}, for every $t$,
$$
\partial_t H^{(n)}(0,t)=\frac{3}{t}H^{(n)}(0,t) + 2D^{(n)}(0,t) + \int_{B_n(t)} \langle u, \mathcal{R}^{(n)} u \rangle + \int_{\partial B_n(t)} \mathfrak{t}^{(n)}|u|^2,
$$
$$
\partial_t H(0,t)=\frac{3}{t}H(0,t) + 2D(0,t) + \int_{\bar B(t)} \langle u, \mathcal{R} u \rangle + \int_{\partial \bar B(t)} \mathfrak{t}|u|^2,
$$
where $\mathcal{R}^{(n)}$ and $\mathfrak{t}^{(n)}$ are bounded terms that are uniformly convergent to $\mathcal{R}$ and $\mathfrak{t}$ as $n$ goes to infinity.
The uniform convergence of $|u_n|$ and $g_n$ then imply
$$
\lim_{s\to\infty} \int_{2\sigma}^{\sigma r}D^{(n)}(0,t)\,dt=\int_{2\sigma}^{\sigma r}D(0,t)\,dt,
$$
which contradicts \eqref{i: compare of int nabla u in geodesic ball}.
In conclusion, 
$$
\lim_{n\to\infty} \int_{\bar B(2)}|\nabla_{A_n}u_n|^2 = \int_{\bar B(2)}|\nabla_{A}u|^2.
$$ 
Since $(A_n,g_n)\to (A,g)$ in $C^\infty$, this implies 
$$\lim_{n\to\infty} \|U_i\|_{W^{1,2}}(\bar B(2)) = \|U\|_{W^{1,2}}(\bar B(2)),$$
 therefore $U_i$ convergence strongly to $U$ in $W^{1,2}(\bar B(2))$.
\end{proof}

\begin{Corollary} \label{cor: compactness for affine case}
Let $\sigma>1$. Let $g_*$ be a metric on $\mathbb{R}^4$ given by a constant metric matrix, such that all eigenvalues of the matrix are in the interval $[\sigma^{-2},\sigma^2]$.

Assume $\{(g_n,A_n,\rho_n)\}_{n\ge 1}$ is a sequence of geometric data on $\bar B(5\sigma^2)$, and assume $(g_n,A_n,\rho_n)$ converge to $(g_*,A,\rho)$ in $C^\infty$. Let $U_n$ be a $\mathbb{Z}/2$ harmonic spinor on $\bar B(5\sigma^2)$ with respect to $(g_n,A_n,\rho_n)$, such that the sequence $U_n$ satisfies conditions (2) to (4) listed before proposition \ref{p: compact}. Then a subsequence of $U_n$ converges to a $\mathbb{Z}/2$ harmonic spinor in $W^{1,2}(\bar B(2))$ with respect to $(g,A,\rho)$. The limit $U$ satisfies assumption \ref{assumption}, and the sequence $U_n$ converges to $U$ uniformly.
\end{Corollary}

\begin{proof}
Take a linear map $T:\mathbb{R}^4\to\mathbb{R}^4$ such that $T^*(g_*)$ is the Euclidean metric. Then $(T^*g_n,T^*A_n,T^*\rho_n,T^*U_n)$ gives a sequence of $\mathbb{Z}/2$ harmonic spinor on $\bar B(5\sigma)$. Since $T^*g_n$ converges to the Euclidean metric, one can apply lemma \ref{p: compact} and find a convergent subsequence on $\bar B(2\sigma)$. Now pull back by $T^{-1}$, one obtains a convergent subseqence of $U_n$ on $\bar B(2)$.
\end{proof}

\section{Frequency pinching estimates}
For $x\in B_{x_0}(32R)$ and $0<s<r\le 32R$, define
$$
W_{s}^{r}(x)= N_\phi(x,r)-N_\phi(x,s).
$$
This section proves the following estimate
\begin{Proposition} \label{prop: pinching}
There exists a constant $C$ with the following property. Let $r\in (0,8R]$. Assume $x_1, x_2\in B_{x_0}(32R)$, such that $d(x_1,x_2)\le r/4$. Let $x$ be a point on the short geodesic $\gamma$ bounded by $x_1$ and $x_2$.  Let $v$ be a unit tangent vector of $\gamma$ at $x$. Then 
$$
d(x_1,x_2)\cdot |\partial_v N_\phi(x,r)| \le C\Big[ \sqrt{|W^{4r}_{r/4}(x_1)|} + \sqrt{|W^{4r}_{r/4}(x_2)|} + r \Big].
$$
\end{Proposition}

The proof is adapted from the arguments in \cite[Section 4]{rec}.
First, one needs to prove the following lemma.
\begin{Lemma} \label{lem: pinching with fixed center}
There exists a constant $C$, such that for every $x\in B_{x_0}(32R)$ and $r\le 8R$, one has
$$
 \int_{B_{x}(3r)-B_x(r/3)}  |\nabla_{\eta_x} u(y) - N_\phi(x,d(x,y)) u(y) |^2 dy\le CrH_\phi(x,r)(W_{r/4}^{4r}(x)+Cr^2).
$$
\end{Lemma}

\begin{proof}
By equation \eqref{eqn: r der of N},
\begin{align*}
&\int_{r/4}^{4r}\partial_s N_\phi(x,s)ds+ O(r^2)
\\ =&
\int_{r/4}^{4r}\frac{2}{sH_\phi(x,s)}\int -\phi'\Big(\frac{d(x,y)}{s}\Big) d(x,y)^{-1} |\nabla_{\eta_x} u(y) - N_\phi(x,s) u(y) |^2 \, dy ds 
\\ \ge &
\frac{1}{C_1rH_\phi(x,r)}\int_{r/4}^{4r}\int -\phi'\Big(\frac{d(x,y)}{s}\Big) d(x,y)^{-1} |\nabla_{\eta_x} u(y) - N_\phi(x,s) u(y) |^2 \, dy ds
\\ \ge &
\frac{1}{C_1rH_\phi(x,r)}\int_{r/3}^{4r}\int -\phi'\Big(\frac{d(x,y)}{s}\Big) d(x,y)^{-1} |\nabla_{\eta_x} u(y) - N_\phi(x,s) u(y) |^2 \, dy ds
\\ =&
: (A)
\end{align*}
For every pair $(y,s)$ in the support of the integration in $(A)$, one has $d(x,y)\in[r/4,4r]$, hence
$$
|N_\phi(x,s)-N_\phi(x,d(x,y))| \le W_{r/4}^{4r}(x)+C_2 r^2.
$$
Therefore,
\begin{align*}
&
(A)\ge \frac{1}{C_1rH_\phi(x,r)}\underbrace{\int_{r/3}^{4r}\int -\phi'\Big(\frac{d(x,y)}{s}\Big) d(x,y)^{-1} |\nabla_{\eta_x} u(y) - N_\phi(x,d(x,y)) u(y) |^2 \, dy ds}_{=:I}
\\ &
-\frac{C_3 (W_{r/4}^{4r}(x)+C_2 r^2)}{rH_\phi(x,r)} \underbrace{\int_{r/3}^{4r} \int -\phi'\Big(\frac{d(x,y)}{s}\Big) d(x,y)^{-1} \Big[|\nabla u(y)||u(y)| d(x,y) + |u(y)|^2 \Big]dy ds}_{=:II}.
\end{align*}
By lemma \ref{lem: bound |u||nabla u|},  $II=O(rH_\phi(x,4r))=O((rH_\phi(x,r))$. By Fubini's theorem,
$$
I = \int_{B_{x}(4r)}  |\nabla_{\eta_x} u(y) - N_\phi(x,d(x,y)) u(y) |^2 \int_{r/3}^{4r} -\phi'\Big(\frac{d(x,y)}{s}\Big) d(x,y)^{-1}  \, ds dy
$$
Notice that 
$$
\inf_{\{y|d(x,y)\in[r/3,3r]\}} \int_{r/3}^{4r}-\phi'\Big(\frac{d(x,y)}{s}\Big) d(x,y)^{-1}  \, ds >0,
$$
Therefore
$$
I\ge \frac{1}{C_4}\int_{B_{x}(3r)-B_x(r/3)}  |\nabla_{\eta_x} u(y) - N_\phi(x,d(x,y)) u(y) |^2\,dy,
$$
In conclusion,
\begin{multline*}
(A)\ge  \frac{1}{C_5 rH_\phi(x,r)}\int_{B_{x}(3r)-B_x(r/3)}  |\nabla_{\eta_x} u(y)
- N_\phi(x,d(x,y)) u(y) |^2\,dy
\\
-
C_6(W_{r/4}^{4r}(x)+C_2r^2),
\end{multline*}
hence
$$
C_7rH_\phi(x,r)(W_{r/4}^{4r}(x)+C_8 r^2) \ge \int_{B_{x}(3r)-B_x(r/3)}  |\nabla_{\eta_x} u(y) - N_\phi(x,d(x,y)) u(y) |^2 dy.
$$
\end{proof}

One also needs the following technical lemma.

\begin{Lemma} \label{l: Euclidean approx}
Assume $M$ is a compact manifold, possibly with boundary. Let $\varphi^\zeta: \Omega\subset \overline{B_{x_0}(64R)}\to\mathbb{R}^4$ be a smooth family of smooth embeddings, parametrized by $\zeta\in M$. For every $\zeta\in M$ and $x\in B_{x_0}(64R)$, one can define a vector field $\eta^\zeta_x$ on $B_{x_0}(64R)$ as follows. For every $y\in B_{x_0}(64R)$, let 
$$\eta^\zeta_x(y)=[(\varphi^\zeta)_*(y)]^{-1}(\varphi^\zeta(y)-\varphi^\zeta(x)).$$
Then there exists a constant $\Theta>0$, depending on $\varphi$, such that
$$
|\eta^\zeta_x(y)-\eta_x(y)|\le \Theta\cdot  d(x,y)^2.
$$
\end{Lemma}
\begin{proof}
Fix $x$, compute the covariant derivates of $\eta^\zeta_x$ and $\eta_x$ at $x$. Since both vector fields are zero at $x$, their covariant derivatives at $x$ are independent of the connections. Let $e\in T_x X$. Taking derivate in the Euclidean coordinates $\varphi^\zeta$, one obtains $\nabla_e(\eta^\zeta_x)(x)=e$. Taking derivative in the normal coordinates centered at $x$, one obtains $\nabla_e (\eta_x)(x)=e$. Therefore, $\eta^\zeta_x$ and $\eta_x$ have the same derivatives at $x$. Since we are working on compact manifolds, $|\eta^\zeta_x(y)-\eta_x(y)|\le \Theta\cdot  d(x,y)^2$ for some constant $\Theta$ independent of $x$.
\end{proof}

\begin{proof}[Proof of proposition \ref{prop: pinching}]
Assume that $v$ points from $x_1$ towards $x_2$. Extend $v$ to a vector field on $B_x(r)$, such that the coordinates of $v$ are constant under the normal coordinate centered at $x$.
Now apply lemma \ref{l: Euclidean approx}. Let $M=\overline{B_{x_0}(32R)}$. For every $\zeta\in \overline{B_{x_0}(32R)}$, let $\varphi^\zeta$ be the exponential map centered at $\zeta$. Then for every $z\in B_x(r)$,
\begin{equation}\label{i: compare with Euclidean 1}
v(z)=\frac{\eta_{x_1}^x(z)-\eta_{x_2}^x(z)}{|\varphi^x(x_1)-\varphi^x(x_2)|}.
\end{equation}
By lemma \ref{l: Euclidean approx}, 
\begin{equation}\label{i: compare with Euclidean 2}
|\eta_{x_1}^x(z)-\eta_{x_1}(z)|=O(r^2),\quad |\eta_{x_2}^x(z)-\eta_{x_2}(z)|=O(r^2)
\end{equation}
Notice that since $\varphi^x$ is the exponential map centered at $x$,
\begin{equation}\label{i: compare with Euclidean 3}
|\varphi^x(x_1)-\varphi^x(x_2)|=d(x_1,x_2).
\end{equation}
Combine \eqref{i: compare with Euclidean 1}, \eqref{i: compare with Euclidean 2} and \eqref{i: compare with Euclidean 3} together, one obtains
$$
\Big|v(z)-\frac{\eta_{x_1}(z)-\eta_{x_2}(z)}{d(x_1,x_2)}\Big|=O(r^2/d(x_1,x_2)).
$$
Define
$$
\mathcal{E}_l(z)=\nabla_{\eta_{x_l}} u(z) - N_\phi(x_l,d(z,x_l)) u(z) \qquad \text{for } l=1,2.
$$
Then
\begin{align*}
d(x_1,x_2)\nabla_{v} u(z) = & \nabla_{\eta_{x_1}} u(z) - \nabla_{\eta_{x_2}} u(z) + O(r^2 |\nabla u|)\\
=& \underbrace{\big(N_\phi(x_1,d(z,x_1)) - N_\phi(x_2,d(z,x_2))\big)}_{=:\mathcal{E}_3(z)}\,u(z)
\\
& + \mathcal{E}_1(z) - \mathcal{E}_2(z)+ O(r^2 |\nabla u|).
\end{align*}

To simplify notations, define the measure
$$
d\mu_x = -d(x,y)^{-1}\phi'\Big(\frac{d(x,y)}{r}\Big) \, dy.
$$

Using \eqref{eqn: v der of N}, one can write
\begin{align*}
&d(x_1,x_2)\cdot\partial_v N_\phi(x,r) 
\\=& 
O(r^2) +\frac{2}{H_\phi(x,r)}\int \nabla_{\eta_x} u(y) \cdot (\mathcal{E}_1- \mathcal{E}_2 + \mathcal{E}_3 u+ O(r^2|\nabla u|))d\mu_x 
\\&
-\frac{2}{H_\phi(x,r)} \int uN_\phi(x,r)\cdot (\mathcal{E}_1- \mathcal{E}_2 + \mathcal{E}_3 u + O(r^2|\nabla u|))d\mu_x 
\\ = &
\underbrace{\frac{2}{H_\phi(x,r)}\int \nabla_{\eta_x} u(y) \cdot (\mathcal{E}_1- \mathcal{E}_2 )d\mu_x}_{=:(A)} -\underbrace{\frac{2N_\phi(x,r)}{H_\phi(x,r)} \int u\cdot (\mathcal{E}_1- \mathcal{E}_2  )d\mu_x}_{=:(B)}
\\ &
+\underbrace{\frac{2}{H_\phi(x,r)}  \int \mathcal{E}_3 u(\nabla_{\eta_x} u -  N_\phi(x,r) u) \, d\mu_x}_{=:(C)} +O(r^2)
\end{align*}

To bound $(C)$, notice that
\begin{align*}
\mathcal{E}_3(z) &=
\underbrace{N_\phi(x_1, r) - N_\phi(x_2, r)}_{=:\mathcal{E}} +
\underbrace{[N_\phi(x_1, d(z,x_1)) - N_\phi(x_1, r) ] }_{=:\mathcal{E}_4(z)}
\\ 
&- \underbrace{[N_\phi(x_2, d(z,x_2)) -N_\phi(x_2, r)]}_{=:\mathcal{E}_5(z)}. 
\end{align*}
By \eqref{eqn: int by parts 1},
\begin{align*}
\int u\cdot \nabla_{\eta_x} u \,d\mu_x & = rD_\phi(x,r)+O(r^2H_\phi(x,r))
\\
& = N_\phi(x,r)H_\phi(x,r)+O(r^2H_\phi(x,r))
\\
& = N_\phi(x,r)\int |u|^2\,d\mu_x+O(r^2H_\phi(x,r)).
\end{align*}
Hence
$$
\int u\cdot (\nabla_{\eta_x} u -  N_\phi(x,r) u) d\mu_x
\\= O(r^2H_\phi(x,r)),
$$
therefore
$$
\int \mathcal{E}u\cdot (\nabla_{\eta_x} u -  N_\phi(x,r) u) d\mu_x
\\= O(r^2H_\phi(x,r)).
$$
By lemma \ref{lem: bound |u||nabla u|},
$$
2 \int  |u|(|\nabla_{\eta_x} u| +  |N_\phi(x,r)| |u|) \, d\mu_x = O(H_\phi(x,r)).
$$
In addition, notice that
$$
\sup_{z\in \text{ supp }\mu_x }|\mathcal{E}_4(z)|+|\mathcal{E}_5(z)| \le W^{4r}_{r/4}(x_1) + W^{4r}_{r/4}(x_2) + C_1 r^2. 
$$
Therefore,
\begin{multline*}
\int (|\mathcal{E}_4|+|\mathcal{E}_5|) \cdot \big|u(\nabla_{\eta_x} u -  N_\phi(x,r) u)\big| \, d\mu_x
\\
\le C_2H_\phi(x,r)(W^{4r}_{r/4}(x_1) + W^{4r}_{r/4}(x_2) + C_1 r^2).
\end{multline*}
As a result,
$$
(C) \le C_3(W^{4r}_{r/4}(x_1) + W^{4r}_{r/4}(x_2)+C_4r^2).
$$

To bound $(A)$, use Cauchy's inequality to obtain
\begin{align*}
(A) \le & \frac{C_5}{H_\phi(x,r)} \Big(\int_{B_x(r)} |\nabla u|^2 dy \Big)^{1/2}  \Big(\int_{B_x(r)-B_x(3r/4)} \big(\mathcal{E}_1^2+ \mathcal{E}_2^2\big)dy \Big)^{1/2}
\\
\le & \frac{C_6}{r^{1/2}} \Big(\int_{B_x(r)-B_x(3r/4)} \big(\mathcal{E}_1^2+ \mathcal{E}_2^2\big)dy \Big)^{1/2}.
\end{align*}
Now apply lemma \ref{lem: pinching with fixed center}, 
\begin{align*}
\int_{B_x(r)-B_x(3r/4)}\mathcal{E}_1^2\,dy  & \le 
\int_{B_{x_1}(5r/4)-B_{x_1}(r/2)}\mathcal{E}_1^2\,dy 
\\
& \le C_7rH_\phi(x_1,r)(W_{r/4}^{4r}(x_1)+C_7 r^2)
\end{align*}
A similar estimate works for the integral of $\mathcal{E}_2$. Therefore
$$
(A) \le C_8\Big[ \sqrt{|W^{4r}_{r/4}(x_1)|} + \sqrt{|W^{4r}_{r/4}(x_2)|} + r \Big].
$$
Similarly, applying Cauchy's inequality on $(B)$ leads to
\begin{align*}
(B) \le & \frac{C_9}{rH_\phi(x,r)} \Big(\int_{B_x(r)} |u|^2 dy \Big)^{1/2}\Big(\int_{B_x(r)-B_x(3r/4)} \big(\mathcal{E}_1^2+ \mathcal{E}_2^2\big)dy \Big)^{1/2}
\\
\le & \frac{C_{10}}{r^{1/2}} \Big(\int_{B_x(r)-B_x(3r/4)} \big(\mathcal{E}_1^2+ \mathcal{E}_2^2\big)dy \Big)^{1/2}
\end{align*}
Lemma \ref{lem: pinching with fixed center} then gives 
$$
(B) \le C_{11}\Big[ \sqrt{|W^{4r}_{r/4}(x_1)|} + \sqrt{|W^{4r}_{r/4}(x_2)|} + r \Big],
$$
and the proposition is proved.
\end{proof}

\begin{Corollary} \label{c: pointwise pinching}
Assume $x_1, x_2\in B_{x_0}(32R)$, assume $r\in (0,8R]$. If $d(x_1,x_2)\le r/4$, then 
$$
|N_\phi(x_1,r)-N_\phi(x_2,r)| \le C\Big[ \sqrt{|W^{4r}_{r/4}(x_1)|} + \sqrt{|W^{4r}_{r/4}(x_2)|} + r \Big].
$$
\end{Corollary} \qed

\section{$L^2$ approximation by planes}\label{s: flatness}

This section establishes a distortion bound in the spirit of \cite{naber2015rectifiable}.
Assume $U$ satisfies assumption \ref{assumption} with respect to $\epsilon>0$. In this section, the constants $C$, $C_1$, $C_2$, $\cdots$ will denote constants that depend on $\Lambda$, $R$, the $C^1$ norms of the curvatures, as well as $\epsilon$.  The techniques in this section were developed by \cite{naber2015rectifiable}, and the presentation here is adapted from section 5 of \cite{rec}.

\begin{Definition}
Suppose $\mu$ is a Radon measure on $\mathbb{R}^4$. For $x\in\mathbb{R}^4$, $r>0$, define
$$
D_\mu^2(x,r)= \inf_L r^{-4} \int_{B_x(r)} \text{dist}(y,L)^2 \,d\mu(y),
$$
where $L$ is taken among the set of $2$-dimensional affine subspaces.
\end{Definition}

For a measure $\mu$ supported in $Z$, we wish to bound the value of $D_\mu^2(x,r)$ in terms of the frequency functions. However, we have to be careful, since $X$ is a Riemannian manifold, but $D_\mu^2(x,r)$ is only defined for Euclidean spaces. We identify $B_{x_0}(32R)$ with $\bar B(32R)$ using the exponential map centered at $x_0$. From now on, we will work on the Euclidean space using this identification.

The main result of this section is the following
\begin{Proposition} \label{p: estimate of D}
There exists a positive constant $R_0\le R$ and a constant $C$ with the following property. Let $\mu$ be a Radon measure supported in $Z$. For $x\in  \bar B(R)$ and $r\le R_0$, one has
$$
D_\mu^2(x,r/8)\le \frac{C}{r^2} \int_{\bar B_x(r/8)} (W^{4r}_{r/4}(z) +C r^2)d\mu(z).
$$
\end{Proposition}

First, observe that the function $D_\mu^2(x,r)$ has the following geometric interpretation. Assume $\mu(\bar B_x(r))>0$, let
$$
\bar z = \frac{1}{\mu(\bar B_x(r))} \int_{\bar B_r(x)} z \, d\mu(z),
$$
Define a non-negative bilinear form $b$ on $\mathbb{R}^4$ as
$$
b(v,w)= \int_{\bar B_x(r)} \big((z-\bar z)\cdot v\big) \big((z-\bar z)\cdot w\big) \,d\mu(z).
$$
Let $0\le \lambda_1\le\cdots \le \lambda_4$ be the eigenvalues of $b$, then
$$
D_\mu^2(x,r)=r^{-4}(\lambda_1+ \lambda_2).
$$
Let $v_i$ be an eigenvector with eigenvalue $\lambda_i$, a straightforward argument of linear algebra shows that
\begin{equation}\label{e: eigenvalue}
\int_{B_x(r)} \big((z-\bar z)\cdot v_i\big) z \,d\mu(z)=\lambda_i\, v_i.
\end{equation}

The following lemma can be understood as a version of Poincar\'e inequality for $\mathbb{Z}/2$ harmonic spinors.
\begin{Lemma} \label{l: lower bound on der at d-1 directions}
There exist constants $C, R_0>0$ with the following property. Let $v_1,v_2,v_3$ be orthonormal vectors in $\mathbb{R}^4$. Let $x\in\bar B(R)$, $r\le R_0$. Assume $Z\cap \bar B_x(r/8)\neq \emptyset$, then
$$
\int_{\bar B_x(5r/4)-\bar B_x(3r/4)} \sum_{j=1}^{3} |\nabla_{v_j} u(z)|^2 \, dz \ge \frac{H_\phi(x,r)}{Cr}.
$$
\end{Lemma}
\begin{proof}
Assume such constants do not exist. Then there exists a sequence 
$$\{(x_n,r_n,U_n,v_1^{(n)},v_2^{(n)},v_3^{(n)})\}_{n\ge 1},$$ 
such that $r_n\le \frac{1}{n}$, the vectors $v_1^{(n)},v_2^{(n)},v_3^{(n)}$ are orthonormal in $\mathbb{R}^4$,
\begin{equation}\label{e: u is small in 3 directions}
\int_{\bar B_{x_n}(5{r_n}/4)-\bar B_{x_n}(3r_n/4)} \sum_{j=1}^{3} |\nabla_{v_j^{(n)}} u(z)|^2 \, dz \le \frac{H_\phi(x_n,r_n)}{nr_n},
\end{equation}
and $Z\cap \bar B_{x_n}(r_n/8)\neq \emptyset$.

Let $\sigma=(12/11)^2$. Rescale the ball $\bar B_{x_n}(5\sigma^2 r_n)$ to $\bar B(5\sigma^2)$, and normalize the restriction of $U$. By assumption \eqref{i: close to euclidean}, the pull back metrics $g_n$ are given by matrix-valued functions on $\bar B(5\sigma^2)$ with eigenvalues bounded by $1/\sigma^2$ and $\sigma^2$. There is a subsequence of the pull backs of $(g_n,A_n,\rho_n,v_1^{(n)},v_2^{(n)},v_3^{(n)})$ that converges to some data set $(g,A,\rho,v_1,v_2,v_3)$ in $C^\infty$, and since $r_n\to 0$, the limit data set $(g,A,\rho)$ is invariant under translations. By corollary \ref{cor: compactness for affine case}, after taking a subsequence, the rescaled $U_n$ converges to a $\mathbb{Z}/2$ harmonic spinor $U^*$ on $\bar B(2)$ with respect to $(g,A,\rho)$, which satisfies assumption \ref{assumption}.

The assumption that $Z\cap\bar B_{x_n}(r_n/8)\neq \emptyset$ implies that $U^*$ has at least one zero point in $\bar B(1/8)$. Inequality \eqref{e: u is small in 3 directions} gives 
$$
\int_{\bar B(5/4)-\bar B(3/4)} \sum_{j=1}^{3} |\nabla_{v_j} u^*(z)|^2 \, dz =0
$$
Theorem \ref{thm: Taubes} implies that $U^*$ is not identically zero on $\bar B(5/4)-\bar B(3/4)$.
Since $U^*$ solves the Dirac equation on non-zero points, the unique continuation property implies that $|U|$ is constant in $3$ linearly independent directions in $\bar B(5/4)-\bar B(3/4)$, hence theorem \ref{thm: Taubes} implies that $U$ is everywhere non-zero in $\bar B(5/4)$, and that is a contradiction.
\end{proof}

Now one can give the proof of proposition \ref{p: estimate of D}. The proof is adapted from the proof of proposition 5.3 in \cite{rec}.  

\begin{proof}[Proof of proposition \ref{p: estimate of D}]
Let $R_0$ be given by lemma \ref{l: lower bound on der at d-1 directions}, and assume $r\le R_0$. Without loss of generality, assume that $D_\mu^2(x,r/8)>0$. In particular, $\mu(\bar B_x(r/8))>0$, thus $Z\cap \bar B_{x}(r/8) \neq \emptyset$. Let $$
\bar z =\frac{1}{\mu(\bar B_x(r/8))}\int_{\bar B_x(r/8)} z d\mu(z).
$$
Let $0\le\lambda_1\le\cdots\le \lambda_4$ be the corresponding eigenvalues, then $D_\mu^2(x,r/8)>0$ implies $\lambda_2>0$. Let $v_i$ be the unit eigenvector with eigenvalue $\lambda_i$. Let $\grad \, u(z)$ be the vector in $T_z\mathbb{R}^4\otimes \mathcal{V}$, such that for every $v\in T_z\mathbb{R}^4$,
$$
\langle v, \text{grad}\, u(z)\rangle_{\mathbb{R}^4} = 
\nabla_v u(z).
$$
By \eqref{i: close to euclidean}, $\|\grad \,u(z)\|_{\mathbb{R}^4} \le (\frac{12}{11})\|\nabla u\|_X$.
Equation \eqref{e: eigenvalue} gives
$$
-\lambda_i v_i\cdot  \grad\, u(y) = \int_{\bar B_{x}(r/8)} \big((z-\bar z)\cdot v_i\big) \big( (y-z)\cdot \grad\, u(y) - \alpha u(y)\big) d\mu(z)
$$
for any constant $\alpha$.
By Cauchy's inequality
\begin{align*}
&\lambda_i^2 |v_i\cdot\grad\, u(y)|^2  
\\ 
\le &
\int_{\bar B_{x}(r/8)} \big|(z-\bar z)\cdot v_i\big|^2d\mu(z) \int_{\bar B_{x}(r/8)} \big| (y-z)\cdot \grad\, u(y) - \alpha u(y)\big|^2d\mu(z) 
\\
= &
\lambda_i \int_{\bar B_{x}(r/8)} \big| (y-z)\cdot \grad\, u(y) - \alpha u(y)\big|^2d\mu(y) 
\end{align*}
Therefore, when $\lambda_i\neq 0$,
$$
\lambda_i|v_i\cdot\grad\, u(y)|^2 \le \int_{\bar B_{x}(r/8)} \big| (y-z)\cdot \grad\, u(y) - \alpha u(y)\big|^2d\mu(z).
$$
Integrate with respect to $y$ on $\bar B_x(5r/4)-\bar B_x(3r/4)$, and sum up $i=2,3,4$,
\begin{align}
&\int_{\bar B_x(5r/4)-\bar B_x(3r/4)}\sum_{i=2}^4\lambda_i|v_i\cdot\grad\, u(y)|^2 \, dy 
\nonumber \\ \le &
3\int_{y\in \bar B_x(5r/4)-\bar B_x(3r/4)}\int_{z\in\bar B_{x}(r/8)} \big| (y-z)\cdot \grad\, u(y) - \alpha u(y)\big|^2d\mu(z)dy
\nonumber \\ \le &
3\int_{z\in\bar B_{x}(r/8)} \int_{y\in \bar B_z(11r/8)-\bar B_z(5r/8)}\big| (y-z)\cdot \grad\, u(y) - \alpha u(y)\big|^2\,dyd\mu(z). \label{e: bound eigenbalue by N}
\end{align}
On the other hand,
\begin{align*}
 r^2D_\mu^2(x,r) \sum_{i=2}^4|v_i\cdot\grad\, u(y)|^2 
 = &
r^{-2}(\lambda_1+ \lambda_2) \sum_{i=2}^4|v_i\cdot\grad\, u(y)|^2 
\\ \le &
\frac{2}{r^2}\sum_{i=2}^4\lambda_i|v_i\cdot\grad\, u(y)|^2 
\end{align*}
Therefore 
\begin{multline*}
 r^2D_\mu^2(x,r) \int_{\bar B_x(5r/4)-\bar B_x(3r/4)}\sum_{i=2}^4|v_i\cdot\grad\, u(y)|^2 \,dy
\\ 
\le \frac{2}{r^2}\int_{\bar B_x(5r/4)-\bar B_x(3r/4)}
\sum_{i=2}^4\lambda_i|v_i\cdot\grad\, u(y)|^2 \, dy 
\end{multline*}
By lemma \ref{l: lower bound on der at d-1 directions}, this implies
$$ 
r^2 H_\phi(x,r) D_\mu^2(x,r) \le \frac{C_1}{r}\int_{\bar B_x(5r/4)-\bar B_x(3r/4)}
\sum_{i=2}^4\lambda_i|v_i\cdot\grad\, u(y)|^2 \, dy 
$$
Therefore inequality \eqref{e: bound eigenbalue by N} gives
\begin{samepage}
\begin{multline} \label{i: bound D_mu by N}
r^2 H_\phi(x,r) D_\mu^2(x,r) 
\\
\le \frac{3C_1}{r}\int_{\bar B_{x}(r/8)} 
\underbrace{\int_{\bar B_z(11r/8)-\bar B_z(5r/8)}\big| (y-z)\cdot \grad\, u(y) - \alpha u(y)\big|^2\,dy}_{=:A(z,r)} d\mu(z).
\end{multline}
\end{samepage}
where the constant $C_1$ is independent of $\alpha$.

Notice that
\begin{multline*}
A(z,r)\le 3 \Big( 
\underbrace{\int_{\bar B_z(11r/8)-\bar B_z(5r/8)}\big| \eta_z(y)\cdot \grad\, u(y) - N_\phi(z,d(z,y)) u(y)\big|^2\,dy }_{=:A_1(z,r)}
\\
+\underbrace{\int_{\bar B_z(11r/8)-\bar B_z(5r/8)} |(y-z)-\eta_z(y)|^2|\grad\, u(y)|^2 dy}_{=:A_2(z,r)}
\\
+ \underbrace{\int_{\bar B_z(11r/8)-\bar B_z(5r/8)} \big(N_\phi(z,d(z,y))-\alpha\big)^2 |u(y)|^2dy }_{=:A_3(z,r)}\Big)
\end{multline*}
Notice that by \eqref{i: close to euclidean}, we have
$\bar B_z(11r/8)-\bar B_z(5r/8) \subset B_z(3r/2)-B_z(r/2)$.
Therefore, by lemma \ref{lem: pinching with fixed center},
$$A_1(z,r)\le C_2rH_\phi(z,r)(W_{r/4}^r(z)+C_2r^2).$$

By lemma \ref{l: Euclidean approx} and lemma \ref{lem: bound |u||nabla u|}, 
$$
A_2(z,r)=O(r^4\int_{B_z(3r/2)}|\nabla u|^2)=O(r^3H_\phi(x,r)).
$$
To bound $A_3(z,r)$, first break it into two parts
\begin{multline*}
A_3(z,r) \le C_3\underbrace{\int_{B_z(3r/2)-B_z(r/2)}\big(N_\phi(z,d(z,y))-N_\phi(z,r)\big)^2 |u(y)|^2dy}_{=:A_4(z,r)} \\
+ C_4\underbrace{\int_{B_z(3r/8)-B_z(r/2)} \big(N_\phi(z,r)-\alpha\big)^2 |u(y)|^2dy}_{=:A_5(z,r)}
\end{multline*}
Here the balls $B_z(3r/2)$ and $B_z(r/2)$ are the geodesic balls on $X$, and the measure $dy$ is the volume form of $X$.
The monotonicity of $N_\phi$ implies that
\begin{align*}
A_4(z,r) \le & (W_{r/4}^{4r}(z)+C_5 r^2)\int_{B_z(3r/2)}|u(y)|^2dy
\\ \le & 
C_6rH_\phi(x,r)(W_{r/4}^{4r}(z)+C_5 r^2).
\end{align*}
Now take $p\in B_x(r/8)$, such that
$$
|W_{r/4}^{4r}(p)|=\inf_{q\in B_x(r/8)} |W_{r/4}^{4r}(q)|,
$$
and take $\alpha=N_\phi(p,r)$. Then by lemma \ref{c: pointwise pinching}, for $z\in B_x(r/8)$,
\begin{align*}
A_5(z,r)\le & \int_{B_z(3r/2)-B_z(r/2)} \big(C_7(\sqrt{|W_{r/4}^{4r}(z)|}+\sqrt{|W_{r/4}^{4r}(p)|}+r)\big)^2 |u(y)|^2dy
\\ \le &
C_8 \big(W_{r/4}^{4r}(z)+C_8 r^2\big) \int_{B_z(3r/2)-B_z(r/2)}|u(y)|^2dy
\\ \le &
C_9 rH_\phi(x,r)\big(W_{r/4}^{4r}(z)+C_8 r^2\big) 
\end{align*}
In conclusion, 
$$
A(z,r)\le C_{10}rH_\phi(x,r)\big(W_{r/4}^{4r}(z)+C_{11} r^2\big).
$$
Therefore proposition \ref{p: estimate of D} follows from inequality \eqref{i: bound D_mu by N}.
\end{proof}

\section{Approximate spines}
\begin{Definition}
Given a set of points $\{p_i\}_{i=0}^k\subset \mathbb{R}^4$ and a number $\beta>0$, one says that $\{p_i\}_{i=0}^k$ is $\beta$-linearly independent, if for every $j\in\{0,1,\cdots,k\}$, the distance between $p_j$ and the affine subspace spanned by $\{p_i\}_{i=0}^k\backslash\{p_j\}$ is at least $\beta$.

Given a set $F\subset \mathbb{R}^4$, one says that $F$ $\beta$-spans a $k$-dimsensional affine subspace, if there exit $(k+1)$ points in $F$ that are $\beta$-linearly independent.
\end{Definition}

\begin{Lemma} \label{lem: spine}
If $F$ is a bounded set that does not $\beta$-span a $k$-dimensional affine space, then there exists a $(k-1)$-dimensional affine space $V$, such that $F$ is contained in the $2\beta$-neighborhood of $V$.
\end{Lemma}
\begin{proof}
For $k$ points $\{q_1,\cdots,q_k\}$ in $\mathbb{R}^4$, let $V(q_1,\cdots,q_k)$ be the volume of the $(k-1)$ dimensional simplex spanned by these points. 
Let $\{p_1,\cdots,p_{k}\}\subset F$ be $k$ points in $F$ such that 
\begin{equation} \label{i: assumption on volume}
V(p_1,\cdots,p_{k})\ge \frac{1}{2} \sup_{q_1,\cdots,q_k\in F} V(q_1,\cdots,q_k).
\end{equation}
 If the volume $V(p_1,\cdots,p_k)$ is zero, then $F$ is contained in a $(k-1)$-dimensional affine subspace, and the statement is trivial. 
If the volume is positive, then the set $\{p_1,\cdots,p_{k}\}$ spans a $k-1$ dimensional affine space $V$. If $F$ is contained in the $2\beta$ neighborhood of $V$, then the statement is verified. 
Otherwise, there exists a point $p_{k+1}\in F$, such that the distance of $p_{k+1}$ and $V$ is greater than $2\beta$. Let $d_j$ be the distance between $p_j$ and the affine subspace spanned by $\{p_i\}_{i=0}^{k+1}\backslash\{p_j\}$, then $d_{k+1}\ge 2\beta$. By \eqref{i: assumption on volume}, $2d_j\ge d_{k+1}$ for every $j$. Therefore $\{p_1,\cdots,p_{k+1}\}$ is $\beta$-linearly independent, and that contradicts the assumption on $F$.
\end{proof}

As in section \ref{s: flatness}, use the normal coordinate centered at $x_0$ to identify $B_{x_0}(32R)$ with the ball $\bar B(32R)$ in $\mathbb{R}^4$. Recall that by assumption \eqref{i: close to euclidean}, 
$$\big(\frac{11}{12}\big)^2\le \kappa_{x_0}(z) \le K_{x_0}(z) \le \big(\frac{12}{11}\big)^2,$$
where $\kappa_{x_0}(z)$ and $K_{x_0}(z)$ are the upper and lower bound of the eigenvalues of the metric matrix at $z\in \bar B_{x}(32R)$.

The compactness property of $\mathbb{Z}/2$ harmonic spinors leads to the following lemma. 

\begin{Lemma} \label{l: cor of compactness 1}
Let $\beta,\bar\beta,\tilde \beta\in(0,1)$ be given. Then there exits $\delta>0$, depending on $\beta,\bar\beta$, the upper bound $\Lambda$ of the frequency function, the value of $R$, the curvatures of $X$ and $\mathcal{V}$, and the constant $\epsilon$ in assumption \ref{assumption}, such that the following holds. If $x\in \bar B(R)$, $r\le \delta$, and $\{p_1,p_2,p_3\}$ is a set of $\bar\beta r$-linearly independent points in $\bar B_x(r)$, such that
$$
N_\phi(p_i,2r)- N_\phi(p_i,\tilde \beta r)<\delta \quad i=1,2,3.
$$
Let $V$ be the affine space spanned by $p_1,p_2,p_3$. Then the set $Z\cap \bar B_x(r)$ is contained in the $\beta r$ neighborhood of $V\cap \bar B_x(r)$.
\end{Lemma}

\begin{proof}
Assume such $\delta$ does not exist. Then there exist sequences $\{p_i^{(n)}\}_{i=1}^3$, $x_n$, and $r_n$, such that $r_n\to 0$, the points $\{p_i^{(n)}\}_{i=1}^3$ are contained in $\bar B_{x_n}(r_n)$ and are $\bar \beta r_n$-linearly independent, and 
$$
N_\phi(p_i^{(n)},2r_n)- N_\phi(p_i^{(n)},\tilde\beta r_n)<\frac{1}{n} \quad i=1,2,3,
$$
and there exists $y_n\in Z$ such that the distance from $y_n$ to the affine space spanned by $\{p_i^{(n)}\}_{i=1}^3$ is greater than $ \beta r_n$.

Let $\sigma=12/11$.
Rescale the balls $\bar B_{x_n}(10\sigma^2 r_n)$ to radius $10\sigma^2$, and normalize the section $U$. Corollary \ref{cor: compactness for affine case} then gives a limit section $U^*$ which satisfies the following properties:
\begin{enumerate}
\item $U^*$ is a $\mathbb{Z}/2$ harmonic spinor on $\bar B(4)$, with respect to a translation-invariant metric, the trivial connection on $\mathcal{V}$, and a translation invariant Clifford multiplication. $U^*$ satisfies assumption \ref{assumption}.
\item There exist points $p_1^*,p_2^*,p_3^*\in \bar B(1)$, such that they are $\bar\beta$-linearly independent, and 
\begin{equation}\label{e: raduis der for U_infty 1}
N_\phi(p_i^*,2)- N_\phi(p_i^*,\tilde \beta)=0 \quad i=1,2,3,
\end{equation}
\item 
Let $V^*$ be the affine space spanned by $\{p_i^*\}_{i=1}^3$. There exits a point $q\in \bar B(1)$ in the zero set of $U^*$, such that the distance from $q$ to $V^*\cap \bar B(1)$ is at least $\beta$.
\end{enumerate}

Since $U^*$ is defined on a flat manifold with flat bundle, 
remark \ref{r: flat case} indicates that for $U^*$,
$$
\partial_r N_\phi(x,r)=\frac{2}{rH_\phi(x,r)}\int -\phi'\Big(\frac{d(x,y)}{r}\Big) d(x,y)^{-1} |\nabla_{\eta_x} u(y) - N_\phi(x,r) u(y) |^2 \, dy.
$$
Therefore equation \eqref{e: raduis der for U_infty 1} implies that for $i\in\{1,2,3\}$, the section $U^*$ is homogeneous on $\bar B_{p_i^*} (2)-\bar B_{p_i^*}(\tilde \beta)$ with respect to the center $p_i^*$. The unique continuation property for solutions to the Dirac equation implies that $U^*$ is homogeneous on $\bar B(2)$ with respect to $p_i^*$. An elementary argument (see for example \cite[Lemma 6.8]{rec}) then shows that the section $U^*$ is zero on the affine space $V^*$, and that $U^*$ is invariant in the directions parallel to $V^*$. Therefore, property (3) of $U^*$ implies that $U^*$ is zero on a $3$-dimensional affine subspace, which contradicts theorem \ref{thm: Taubes}.
\end{proof}

Similarly, one has
\begin{Lemma} \label{l: cor of compactness 2}
Let $\beta,\bar\beta,\tilde\beta\in(0,1)$ and $\tau>0$ be given. Then there exits $\delta>0$, depending on $\beta,\bar\beta,\tilde\beta,\tau$, the upper bound $\Lambda$ of the frequency function, the value of $R$, the curvatures of $X$ and $\mathcal{V}$, and the constant $\epsilon$ in assumption \ref{assumption}, such that the following holds. Assume $x\in \bar B(R)$, and $r\le \delta$, and $\{p_1,p_2,p_3\}$ is a set of points in $\bar B_x(r)$ that is $\bar \beta r$-linearly independent, such that
$$
N_\phi(p_i,2r)-N_\phi(p_i,\tilde\beta r)<\delta \quad i=1,2,3.
$$
Let $V$ be the affine space spanned by $\{p_i\}$. Then for all $y,y'\in \bar B_x(r)\cap Z$, one has 
$$
|N_\phi(y,\beta r)-N_\phi(y',\beta r)|< \tau.
$$
\end{Lemma}

\begin{proof}
Assume such $\delta$ does not exist, then arguing as before, one obtains a $2$-valued section $U^*$ on $\bar B(4)$ with the following properties:

\begin{enumerate}
\item $U^*$ is a $\mathbb{Z}/2$ harmonic spinor on $\bar B(4)$, with respect to a translation-invariant metric, the trivial connection on $\mathcal{V}$, and a translation invariant Clifford multiplication. $U^*$ satisfies assumption \ref{assumption}.
\item There exist points $p_1^*,p_2^*,p_3^*\in \bar B(1)$, such that they are $\bar \beta$-linearly independent, and 
\begin{equation}\label{e: raduis der for U_infty}
N_\phi(p_i^*,2)-N_\phi(p_i^*, \tilde\beta)=0 \quad i=1,2,3,
\end{equation}
\item 
Let $Z^*$ be the zero set of $U^*$. There exist $y,y'\in \bar B(1)\cap Z^*$, such that 
$$
|N_\phi(y,\beta)-N_\phi(y',\beta)|\ge \tau.
$$
\end{enumerate}
However, as in the proof of the previous lemma, the first two properties imply that $U^*$ is invariant in the directions parallel to the plane $V^*$ spanned by $p_1^*$, $p_2^*$, $p_3^*$, and $Z^*\subset V^*$, which contradicts property (3).
\end{proof}

\section{Rectifiability and the Minkowski bound}

This section only concerns estimates on the Euclidean space. To simplify notations, for the rest of this section, use $B_x(r)$ and $B(r)$ to denote the Euclidean balls.

\begin{Definition} \label{def: tame}
Let $\mathcal{Z}$ be a Borel subset of $\bar B(R)\subset \mathbb{R}^4$. A function $\mathcal{I}(x,r)$ defined for $x\in \mathcal{Z}$ and $r\le 128R$ is called a taming function for $\mathcal{Z}$, if the following conditions hold.
\begin{enumerate}
\item $\mathcal{I}(x,r)$ is non-negative, bounded, continuous, and is non-decreasing in $r$.
\item \label{con: pinching}
Let $\beta,\bar\beta \in(0,1)$ and $\tau>0$ be given. Then there exists $\epsilon(\beta,\bar \beta,\tau)>0$, depending on $\beta,\bar\beta,\tau$, such that the following holds. Assume $x\in \bar B(R)$, $r\le R$, and $\{p_1,p_2,p_3\}$ is a set of points in $\bar B_x(r)$ that is $\bar\beta r$-linearly independent, such that
$$
\mathcal{I}(p_i,2r)-\mathcal{I}(p_i,\beta r/2)<\epsilon(\beta,\bar \beta,\tau) \quad i=1,2,3.
$$
Then for all $y,y'\in \bar B_x(r)\cap \mathcal{Z}$, one has 
$$
|\mathcal{I}(y,\beta r/2)-\mathcal{I}(y',\beta r/2)|< \tau.
$$
\item \label{con: Reifenberg}
There exists a constant $C$, such that for every Radon measure $\mu$ supported in $\mathcal{Z}$, the following inequality holds for every $x\in \bar B(2R)$ and $r\le 2R$:
$$
D_\mu^2(x,r) \le \frac{C}{r^2}\int_{\bar B_x(r)}[\mathcal{I}(z,32r)-\mathcal{I}(z,2r)]\,d\mu(z).
$$
\end{enumerate}
\end{Definition}

The following result follows almost verbatim from sections 7 and 8 of \cite{rec}, and a large part of the arguments originated from \cite{naber2015rectifiable}. Nevertheless, a proof is provided here for the reader's convenience.

\begin{Theorem}[\cite{naber2015rectifiable}, \cite{rec}] \label{thm: covering argument}
Assume $\mathcal{Z}$ is a Borel subset of $ B(R)$ and there exists a taming function $\mathcal{I}(x,r)$ for $\mathcal{Z}$. Then the set $\mathcal{Z} \cap  B(R/2)$ is $2$-rectifiable and has finite $2$-dimensional Minkowski content.
\end{Theorem}

The proof of theorem \ref{thm: covering argument} makes use of the following Reifenberg-type theorem. We state the theorem for the cases of dimension 4 and codimension 2.

\begin{Theorem}[\cite{naber2015rectifiable}, Theorem 3.4]
\label{thm: NV}
There exist universal constants $K_0>0$ and $\delta_0>0$ such that the following holds. Assume $\{B_{x_i}(r_i)\}$ is a collection of balls in $B(2R)$, such that $\{B_{x_i}(r_i/4)\}$ are disjoint. Define a measure $\mu= \sum_i r_i^2\delta_{x_i}$. Suppose
$$
\int_{B_x(r)} \int_0^r \frac{D_\mu^2(z,s)}{s}\,ds d\mu(z)< \delta_0 r^2
$$ 
for every $B_x(r)\subset B(2R)$, then $\mu (B(R))\le K_0 R^2$.
\end{Theorem}

\begin{proof}[Proof of theorem \ref{thm: covering argument}]

Assume $ B_x(r)\subset  B(R)$. If one rescales $ B_x(r)$ to $ B(R)$, then the function $\mathcal{I}'(y,s)= \mathcal{I}(x+(ry)/R, sr/R)$ is a taming function for $[(A-x)\cdot(R/r)]\cap  B(R)$ with the same function $\epsilon(\beta,\bar \beta, \tau)$ and constant $C$. Therefore definition \ref{def: tame} is invariant under rescaling, thus one only needs to consider the case for $R=2$. 

Let $\beta=1/10$. Let $ \bar \beta\le 1/100$ be a positive universal constant, let $\tau>0$ be a constant that is defined by $\bar \beta$ and $C$, and let $\delta>0$ be a constant that is defined by $\bar \beta,\tau$, the function $\epsilon$ and the constant $C$. The exact values for $ \bar \beta,\tau$ and $\delta$ will be determined later in the proof. 

Let $\Lambda$ be an upper bound of $\mathcal{I}$, namly $\displaystyle \Lambda\ge\sup_{x\in A,x\le 128R} \mathcal{I}(x,r)=\sup_{x\in A} \mathcal{I}(x,256)$. 

Define
$$D_\delta(r)=B(R/2)\cap \{x\in \mathcal{Z}| \mathcal{I}(x, \beta r/2)\ge \Lambda-\delta\}.$$
Define 
$$W_{r_1}^{r_2}(x)=\mathcal{I}(x,r_1)-\mathcal{I}(x,r_2).$$
If $\{B_{x_i}(r_i)\}$ is a family of balls, we call the sum $ \sum_i r_i^2$ its 2-dimensional volume.
\\

{\bf Step 1.}
First, require that $\delta<\epsilon(\beta,\bar \beta,\tau)$.
For $ B_x(r)\subset  B(2)$, and a set $A\subset \mathcal{Z}\cap B_x(r)$, define an operator $\mathcal{F}_A$, which turns $B_x(r)$ into a finite set of balls. It has the property that either $\mathcal{F} _A(B_x(r))=\{B_x(r)\}$, or $\mathcal{F} _A(B_x(r))$ is a family of balls with radius $\beta r$. In either case, the balls in $\mathcal{F}(B_x(r))$ will cover the set $A$. The operator $\mathcal{F} _A $ is defined as follows. If $A\cap D_\delta(r)$ does not $\bar\beta r$-span a 2-dimensional affine space, then it is called ``bad''. Otherwise, it is called ``good''. In the bad case, define $\mathcal{F}_A(B_x(r))=\{B_x(r)\}$. In the good case, cover $A$ by a family of balls $\{B_{x_i}(\beta r)\}$ with the following properties
\begin{enumerate}
\item The distance between $x_i$ and $x_j$ is at least $ \beta r/2$ for $\forall i\neq j$,
\item Each $x_i$ is an element of $A$.
\end{enumerate}
Define $\mathcal{F}_A(B_x(r))$ to be the family $\{B_{x_i}(\beta r)\}$. 

Obviouly the descriptions above do not uniquely specify the operator $\mathcal{F}_A$. When there are more than one possibilities, choose one arbitrarily. 

If $B_x(r)$ is a good ball, let $p_1,p_2,p_3\in D_\delta(r)\cap B_x(r)$ be three points that $\bar \beta r$ span a plane, let $\mathcal{F}(B_x(r))=\{B_{x_i}(\beta r)\}$.
By condition (\ref{con: pinching}) of definition \ref{def: tame}, 
$$|\mathcal{I}({x_i},\beta r/2)-\mathcal{I}(p_i,\beta r/2)| \le \tau.$$
Therefore 
\begin{equation} \label{eqn: pinching for small ball}
\mathcal{I}(x_i,\beta r/2)\ge \Lambda-\delta-\tau
\end{equation}

The operator $\mathcal{F}_A$ can be extended to act on a collection of balls. Assume $\{B_{x_i}(r)\}_{i=1}^n$ is a collection of balls with the same radius.
Let $A\subset \bigcup B_{x_i}(r)\cap \mathcal{Z}$. Assume $\{B_{x_i}(r)\}_{i=1}^k$ are the good balls, and $\{B_{x_i}(r)\}_{i=k+1}^n$ are the bad balls. Then there exists a collection of balls $\{B_{y_j}(\beta r)\}$, such that
\begin{enumerate}
\item $\{B_{y_j}(\beta r)\}$ covers 
$\bigcup_{i=1}^k (A\cap B_{x_i}(r)).$
\item $|y_j-y_l|\ge \beta r/2$, for $\forall j\neq l$.
\item $y_j\in \bigcup_{i=1}^k A\cap B_{x_i}(r)$, for $\forall j$.
\end{enumerate}
Inequality \eqref{eqn: pinching for small ball} still holds when $x_i$ is replaced by $y_j$.
Define $\mathcal{F}_A\{B_{x_i}(r)\}$ to be the union of $\{B_{y_j}(\beta r)\}$ and $\{B_{x_i}(r)\}_{i=k+1}^n$. 
\\

{\bf Step 2.} 
Let $N>0$ be a positive integer. Let $A_0(x,r)=\mathcal{Z}\cap B_x(r)$. Apply the operator $\mathcal{F}_{A_0}$ to $B_x(r)$ to obain a set of balls, which we denote by $\mathcal{S}_1(x,r)$. Assume $\mathcal{S}_1(x,r)$ splits to two sets $\mathcal{S}_1(x,r)=\mathcal{S}_{1,g}(x,r)\bigcup \mathcal{S}_{1,b}(x,r)$, where $\mathcal{S}_{1,g}(x,r)$ is the collection of good balls and $\mathcal{S}_{1,b}(x,r)$ is the collection of bad balls. Let 
$$A_1(x,r)=A_0(x,r)-\bigcup_{B_{x_i}(r_i)\in\mathcal{S}_{1,b}(x,r)}B_{x_i}(r_i).$$ Apply $\mathcal{F}_{{A_1}(x,r)}$ to $\mathcal{S}_{1,g}(x,r)$ and obtain a new set of balls 
$$\mathcal{S}_2(x,r)=\mathcal{F}_{A_1(x,r)}(\mathcal{S}_{1,g}(x,r))\bigcup \mathcal{S}_{1,b}(x,r).$$ 
Similarly, write $\mathcal{S}_2(x,r)=\mathcal{S}_{2,g}(x,r)\bigcup \mathcal{S}_{2,b}(x,r)$, and define
$$A_2(x,r)=A_1(x,r)-\bigcup_{B_{x_i}(r_i)\in\mathcal{S}_{2,b}(x,r)}B_{x_i}(r_i),$$
and define $\mathcal{S}_3=\mathcal{F}_{A_2}(\mathcal{S}_{2,g})\bigcup \mathcal{S}_{2,b}$. Repeat the procedure $N$ times to obtain a set of balls $\mathcal{S}_N(x,r)$. 

The family $\mathcal{S}_N(x,r)$ has the following property. If $B_{x_1}(r_1)$ and $B_{x_2}(r_2)$ are two distinct elements of $\mathcal{S}_N(x,r)$, then 
\begin{equation} \label{eqn: separate balls}
|x_1-x_2|\ge (r_1+r_2)/4.
\end{equation}
Inequality \eqref{eqn: separate balls} can be proved by induction. For $N=1$, it follows from the definition of $\mathcal{F}_A$. Assume \eqref{eqn: separate balls} holds for $N-1$, and write $\mathcal{S}_N=\mathcal{F}_{A_{N-1}}(\mathcal{S}_{N-1,g})\bigcup \mathcal{S}_{N-1,b}$. Let $B_{x_1}(r_1),B_{x_2}(r_2)\in\mathcal{S}_N$. If both $B_{x_1}(r_1),B_{x_2}(r_2)\in\mathcal{F}_{A_{N-1}}(\mathcal{S}_{N-1,g})$, then \eqref{eqn: separate balls} follows from the definition of $\mathcal{F}$. If both $B_{x_1}(r_1),B_{x_2}(r_2)\in\mathcal{S}_{N-1,b}$, then \eqref{eqn: separate balls} follows from the induction hypothesis. If $B_{x_1}(r_1)\in\mathcal{F}_{A_{N-1}}(\mathcal{S}_{N-1,g})$, $B_{x_2}(r_2)\in\mathcal{S}_{N-1,b}$, then $x_1\not\in B_{x_2}(r_2)$. By the construction of $\mathcal{F}$, one has $r_1\le \beta r_2$. Since $\beta=1/10$, one has $|x_1-x_2|\ge r_2\ge (r_1+r_2)/2.$

By \eqref{eqn: pinching for small ball}, either $\mathcal{S}_N=\{B_x(r)\}$, or 
\begin{equation}\label{eqn: small ball small pinching}
\mathcal{I}(x_i,r_i/2)\ge\Lambda-\delta-\tau, \quad\forall\, B_{x_i}(r_i)\in\mathcal{S}_N.
\end{equation}

{\bf Step 3.}
We claim that there exists a universal constant $K_1>1$, such that for $\tau$ and $\delta$ sufficiently small, we have
\begin{equation} \label{i: upper bound on measure}
\sum_{B_{x_i}(r_i)\in \mathcal{S}_N(x,r) } r_i^2 < K_1\, r^2. 
\end{equation}
Without loss of generality, assume $\mathcal{S}_N(x,r)\neq\{B_x(r)\}$.
Let $r_j=\beta^{N-j}\,r$. Define Radon measures
$$\mu = \sum_{B_{y}(s)\in \mathcal{S}_N(x,r)} s^2 \delta_{y},$$
$$
\mu_j = \sum_{B_{y}(s)\in \mathcal{S}_N(x,r), s\le r_j} s^2 \delta_{y}.
$$

Notice that by \eqref{eqn: separate balls}, there exists a universal constant $K_2$ such that 
\begin{equation} \label{eqn: j=0}
\mu_0((B_x(r_0))\le K_2 \, r_0^2,\quad \forall x.
\end{equation}

Let $K_0$ be the constant given by theorem \ref{thm: NV}, let $K_3=\max\{K_0,K_2\}$. We prove that if $\tau, \delta$ is chosen sufficiently small, then for every $j=0,1,\cdots,N-3$, and every $B_{y}(r_j)\subset B_x(2r)$, one has 
\begin{equation} \label{eqn: measure bound}
\mu_j(B_{y}(r_j))\le K_3\, r_j^2.
\end{equation}
The claim is proved by induction on $j$. The case for $j=0$ follows from \eqref{eqn: j=0}. Assume that the claim is proved for $0,1,\cdots,j$, and $j<N-3$. Then there exists a universal constant $M>1$, such that for every $y\in B_x(2r)$, $k\le{j+1}$, and $s\in[r_{k}/2, 2r_{k}]$,
\begin{equation} \label{eqn: weak bound}
\mu_{k+3}((B_y(s)) \le M\,(K_3+1) \,s^2
\end{equation}
We want to use theorem \ref{thm: NV} and \eqref{eqn: weak bound} to prove
$$\mu_{j+1}((B_y(r_{j+1})) \le  K_3\, r_{j+1}^2, \text{ for }\forall B_y(r_{j+1})\subset B_x(2r).$$
If $\mu_{j+1}(B_y(r_{j+1}))=0$, the inequality is trivial. From now on assume $\mu(B_y(r_{j+1}))>0$. Since $r_{j+1}\le r_{N-3}=r/8$, and $\supp \mu\subset B_x(r)$, we have $B_y(4r_{j+1})\subset B_x(2r)$. 

Notice that for $B_{x_i}(s_i)\in\mathcal{S}_N$, if $\displaystyle t<\min_k|x_i-x_k|$, then 
$$
D_\mu^2(x_i,t)=0.
$$
Define 
$$
\overline W_{2t}^{32t}(x_i)=
\begin{cases}
0 \qquad\qquad \text{if }t< s_i/4, \\
W_{2t}^{32t}(x_i) \,\,\text{ if } t\ge s_i/4.
\end{cases}
$$
Inequality \eqref{eqn: separate balls} and condition (\ref{con: Reifenberg}) of definition \ref{def: tame} gives
\begin{equation} \label{eqn: modified Reifenberg}
D_{\mu}^2(q,t) \le C\int_{B_q(t)}\frac{\overline W_{2t}^{32t}(p)}{t^3}\,d\mu(p) 
\end{equation}
for every $(q,t)$.

For $B_z(s)\subset B_y(2r_{j+1})$, assume $s\in[r_k/2,2r_k]$ for $k\le j+1$. Inequality \eqref{eqn: modified Reifenberg} gives
\begin{align}
&\int_{B_z(s)} \int_0^s \frac{D_{\mu_{j+1}}^2(q,t)}{t}\,dt \,d\mu_{j+1}(q) \nonumber
\\
\le & C\int_{B_z(s)}\int_0^s\int_{B_q(t)}\frac{\overline W_{2t}^{32t}(p)}{t^3}\,d\mu_{j+1}(p)\,dt\,d\mu_{j+1}(q)
\nonumber
\\
\le & C\int_{B_z(s)}\int_0^s\int_{B_q(t)}\frac{\overline W_{2t}^{32t}(p)}{t^3}\,d\mu_{k+3}(p)\,dt\,d\mu_{k+3}(q)
\label{eqn: radius in the integral}
\\
\le & C\int_{B_z(2s)}\int_0^s\int_{B_p(t)} \frac{\overline W_{2t}^{32t}(p)}{t^3}\,d\mu_{k+3}(q)\,ds\,d\mu_{k+3}(p)
\nonumber
\\
\le & CM(K_3+1) \int_{B_z(2s)}\int_0^s \frac{\overline W_{2t}^{32t}(p)}{t}\,dt\,d\mu_{k+3}(p),
\label{eqn: verify discrete Reifenberg condition step 1}
\end{align}
where inequality \eqref{eqn: radius in the integral} follows from \eqref{eqn: separate balls}. For $p\in\supp\mu_{j+1}$, let $s_p$ be the radius of ball in $\mathcal{S}_N$ with center $p$. If $s\ge s_p/4$, then
\begin{align}
&\int_0^s \frac{\overline W_{2t}^{32t}(p)}{t}\,dt
=
 \int_{s_p/4}^s \frac{W_{2t}^{32t}(p)}{t}\,dt
 = \int_{2s}^{32s} \mathcal{I}(p,t) \,dt
  - \int_{s_p/a}^{16s_p} \mathcal{I}(p,t)\, dt \nonumber
\\
&\quad \le  W_{s_p/2}^{32s}(p)\int_{2}^{32} \frac{1}{t}\,dt 
\le \ln(16)\,(\delta+\tau). \label{eqn: bound int W/t}
\end{align}
The last inequality above follows from \eqref{eqn: small ball small pinching}. 
Therefore, the right hand side of \eqref{eqn: verify discrete Reifenberg condition step 1} is bounded by
\begin{align*}
& CM(K_3+1) \int_{B_z(2s)}\int_0^s \frac{\overline W_{2t}^{32t}(p)}{t}\,dt\,d\mu_{k+3}(p) 
\\
\le & CM(K_3+1) \,  \mu_{k+3}(B_z(2s)) \ln(16)\,  (\tau+\delta)
\le 4CM^2(K_3+1)^2\ln(16) (\tau+\delta)\,s^2
\end{align*}
Let $\delta_0$ be the constant given by theorem \ref{thm: NV}. Take 
$$\tau<\frac{\delta_0}{8CM^2(K_3+1)^2\ln(16)},$$
and
$$\delta<\frac{\delta_0}{8CM^2(K_3+1)^2\ln(16)},$$
then the conditions of theorem \ref{thm: NV} are satisfied, therefore 
$\mu_{j+1}((B_y(r_{j+1})) \le  K_0 \,r_{j+1}^2.$ By induction, \eqref{eqn: measure bound} is proved. Inequality \eqref{i: upper bound on measure} then follows from \eqref{eqn: measure bound} by the the case of $j=N-3$.
\\

{\bf Step 4.} By lemma \ref{lem: spine}, the result obtained from the previous steps can be summarized as follows. For any integer $N>0$, and any ball $B_x(r)$, there is a covering of $\mathcal{Z}\cap B_x(r)$ by a family of balls $\mathcal{S}_N(x,r)=\{B_{x_i}(r_i)\}_i$, such that the following properties hold:
\begin{enumerate}
\item The radius of each ball is at least $\beta^N \, r$.
\item For a all $B_{x_i}(r_i)\in \mathcal{S}_N$, either $r_i=\beta^N\, r$, or $r_i=\beta^j\,r$ for some integer $j<N$, and $B_{x_i}(r_i) \cap D_\delta(r_i)$ is contained in the $2 \bar \beta r_i$ neighborhood of a line.
\item $\sum_i r_i^2 \le K_1\, r^2$.
\end{enumerate}
As a consequence,
\begin{Lemma} \label{lem: admissible cover}
There exists a universal constant $K_1>1$, and a constant $\delta$, such that the following property holds.
For any $B_x(r)\subset B(2)$, and $s\in(0,r)$, there exists a covering of $\mathcal{Z}\cap B_x(r)$ by balls $\mathcal{S}=\{B_{x_i}(r_i)\}_i$, such that
\begin{enumerate}
\item The radius of each ball is at least $\beta s$.
\item For a ball $B_{x_i}(r_i)\in \mathcal{S}$, either $r_i\le s$, or $B_{x_i}(r_i) \cap D_\delta(r_i)$ is contained in the $2 \bar \beta r_i$ neighborhood of a line.
\item $\sum_i r_i^2 \le K_1\, r^2$.
\end{enumerate}
\end{Lemma}

{\bf Step 5.}
We prove the following lemma
\begin{Lemma}\label{lem: strong admissible cover}
There exists a universal constant $K_4$, and a constant $\delta$, such that the following property holds.
For any $B_x(r)\subset B(2)$, and $s\in(0,r)$, there exists a splitting of $\mathcal{Z}$ into $\mathcal{Z}=\bigcup_i\mathcal{E}_i$, and a family of balls $\mathcal{S}=\{B_{x_i}(r_i)\}_i$, such that
\begin{enumerate}
\item $\mathcal{E}_i\subset B_{x_i}(r_i)$.
\item The radius of each ball is at least $4\bar\beta s$.
\item For a ball $B_{x_i}(r_i)\in \mathcal{S}$, either $r_i\in[4\bar\beta s, s]$, or $B_{x_i}(r_i) \cap D_\delta(r_i)=\emptyset$ 
\item $\sum_i r_i^2 \le K_4\, r^2$.
\end{enumerate}
\end{Lemma}

\begin{proof}[Proof of lemma \ref{lem: strong  admissible cover}]

Notice that by the assumptions on $\beta$ and $\bar \beta$, we have $4\bar\beta<\beta$.

If $\{B_{x_i}(r_i)\}_i$ is a covering of $\mathcal{Z} \cap B_x(r)$ that satisfies the three properties given by lemma \ref{lem: admissible cover} with respect to $s$, we say that $\{B_{x_i}(r_i)\}_i$ is an $s$-admissible covering of $B_x(r)\cap \mathcal{Z}$. Fix $s>0$, by lemma \ref{lem: admissible cover}, $s$-admissible coverings of $B_x(r)\cap \mathcal{Z}$ exist.

 Let $\{B_{x_i}(r_i)\}$ be an $s$-admissible covering of $B_x(r)\cap \mathcal{Z}$. Let $\mathcal{E}_i=\mathcal{Z}\cap B_{x_i}(r_i)$. Then the family $\{\big(\mathcal{E}_i, B_{x_i}(r_i)\big)\}$ satisfies conditions (1), (2) of lemma \ref{lem: strong admissible cover}, and $\sum_i r_i^2 \le K_1\, r^2$. However, it may not satisfy condition (3). In the following, we will give a procedure to adjust the family, such that at each step the covering still satisfies property (2) of $s$-admissibility, and after finitely many steps of adjustments, the family will satisfy property (3) of lemma \ref{lem: strong admissible cover}. At the same time, $\sum_i r_i^2$ is being contorlled throughout the adjustments.
 
Assume $\{B_{x_i}(r_i)\}$ is an $s$-admissible covering of $B_x(r)\cap \mathcal{Z}$, and $\mathcal{E}_i\subset B_{x_i}(r_i)$, $B_x(r)\cap \mathcal{Z}=\bigcup \mathcal{E}_i$. Assume $\big(\mathcal{E}_0,B_{x_0}(r_0)\big)$ does not satisfy property (3) of lemma \ref{lem: strong admissible cover}. Then $r_0>s$. 

By property (2) of $s$-admissibility, 
 $B_{x_0}(r_0)\cap D_\delta(r_0)$ is contained in the $2\bar\beta r_0$ neighborhood of a line. Thus one can cover $B_{x_0}(r_0)\cap D_\delta(r_0)$ by a family of no more than $[10/\bar \beta]$ balls with radius $4\bar\beta r_0$. Let $\{B_{y_j}(t_j)\}$ be this family. If $4\bar \beta r_0> s $, apply lemma \ref{lem: admissible cover} again to each ball $B_{y_j}(t_j)$ and replace it with an $s$-admissible covering of $B_{y_j}(t_j)\cap D_\delta(r_0)$. Otherwise keep the family $\{B_{y_j}(t_j)\}$ as it is. Let $\{B_{z_j}(l_j)\}$ be the result of this procedure. Then $\{B_{z_j}(l_j)\}$ covers $B_{x_0}(r_0)\cap D_\delta(r_0)$, and it has the following properties
\begin{enumerate}
\item $4\bar\beta s\le l_j\le 4\bar\beta r_0$ for each $j$,
\item $\sum_j l_j^2 \le [10/\bar \beta]\cdot K_1\,(4\bar \beta r_0)^2.$
\end{enumerate}
Take $\bar \beta\le 1/(320 K_1)$, then $\sum_j l_j^2 \le \frac{1}{2}r_0^2$. 

The adjustment of the family $\{\big(\mathcal{E}_i,B_{x_i}(r_i)\big)\}$ is defined as follows. First, remove $(\mathcal{E}_0,B_{x_0}(r_0))$ from the family, and add $(\mathcal{E}_0\backslash D_\delta(r_0),B_{x_0}(r_0))$ into the family. 
Next, add the family $\{\big(\mathcal{E}_0\cap B_{z_j}(l_j), B_{z_j}(l_j)\big)\}$ constructed from the previous paragraph into this family.

This adjustment replaces an element $(\mathcal{E}_0,B_{x_0}(r_0))$ which does not satisfy property (3) of lemma \ref{lem: strong admissible cover} by a family of balls, such that the biggest ball in this family has the same radius $r_0$ and satisfies property (3). The rest of the balls have radius in the interval $[4\bar\beta s,4\bar\beta r_0]$ and their 2-dimensional volume is bounded by $\frac{1}{2}r_0^2$. Moreover, the new family still satisfies property (2) of lemma \ref{lem: admissible cover}. Therefore, after finitely many times of adjustments, we will obtain a family that satisfies conditions (1), (2), (3), with 2-dimensional volume 
$$
\sum_i r_i^2 \le 2K_1\, r^2,
$$
hence the lemma is proved.
\end{proof}

{\bf Step 6.}
Given $s\in(0,1)$, we use lemma \ref{lem: strong admissible cover} to construct a covering of $\mathcal{Z}\cap B(1)$ by a family of balls $\{B_{x_i}(r_i)\}$ with radius $r_i\in[4\bar\beta s,s]$, such that the 2-dimensional volume of the covering is bounded.

We call a family $\{(\mathcal{E}_i,B_{x_i}(r_i))\}$ a split-covering of a set $A$, if $\mathcal{E}_i\subset B_{x_i}(r_i)$, and $A=\bigcup \mathcal{E}_i$.

If a split-covering of $\mathcal{Z}\cap B_x(r)$ satisfies the properties given by lemma \ref{lem: strong admissible cover}, we say that it is strongly $s$-admissible. 

Let $\mathcal{S}$ be a strongly $s$-admissible split-covering of $\mathcal{Z}\cap B(1)$. 
For every $B_{x_i}(r_i)\in\mathcal{S}$, if $r_i\le s$, we say it is of type I. Otherwise, we say it is of type II. Assume $B_{x_i}(r_i)$ is a ball of type II, then the function $\mathcal{I}(x,r)$ is at most $\Lambda-\delta$ for $x\in \mathcal{E}_i$, $r_i\le \beta r_i/2$. 
There exists a universal constant $L$ such that $\mathcal{E}_i$ can be covered by $L$ balls $B_{y_j}(\beta r_i/512)$ with radius $(\beta r_i/512)$. Therefore, for each ball, the set $\mathcal{E}_i\cap B_{y_j}(\beta r_i/512)$ has a strongly $s$-admissible split-covering, with $\Lambda$ replaced by $\Lambda-\delta$.

Change $(B_{x_i}(r_i),\mathcal{E}_i)$ to the union of the $L$ strongly $s$-admissible split-coverings of 
$\mathcal{E}_i\cap B_{y_j}(\beta r_i/512)$,
we obtain a split-covering of $\mathcal{E}_i$ with 2-dimensional volume at most $LK_4(\beta r_i/512)^2$. Define an operation $\mathcal{G}$ on $\mathcal{S}$, such that $\mathcal{G}(\mathcal{S})$ is constructed from $\mathcal{S}$ by replacing every type II element in $\mathcal{S}$ with the union of the $L$ split-coverings described above.

Notice that for the balls $B_{y_j}(\beta r_i/512)$, the upper bound $\Lambda$ is replaced by $\Lambda-\delta$. Therefore, this procedure can only be carried for at most $N=\lceil\frac{\Lambda}{\delta}\rceil$ times. After that, every ball in $\mathcal{G}^{(N)}(\mathcal{S})$ is of type I. Namely, every ball in $\mathcal{G}^{(N)}(\mathcal{S})$ has radius in the interval $[4\bar\beta s,s]$. 

Let $V_n$ be the 2 dimensional volume of $\mathcal{G}^{(n)}(\mathcal{S})$, then we have 
$$V_{n+1}\le (1+LK_4(\beta/512)^2)V_n.$$ Therefore the total 2-dimensional volume of $\mathcal{G}^{(n)}(\mathcal{S})$ is bounded by
$$V_n\le  (1+LK_4(\beta/512)^2)^NK_4.$$

Since $s$ can be taken to be arbitrarily small, the Minkowski content of $\mathcal{Z}\cap B(1)$ is bounded by a contant $K$ depending on $\Lambda$, $\epsilon$ and $C$.

By rescaling, we conclude that the Minkowski content of $\mathcal{Z}\cap B_x(r)$ is bounded by $K\,r^2$. Since the Minkowski content bounds the Hausdorff measure, there exists a constant $K'$ depending on $\Lambda$, $\epsilon$ and $C$, such that
\begin{equation}\label{eqn: bound on measure}
\mathcal{H}_2(\mathcal{Z}\cap B_x(r))\le K'\,r^2.
\end{equation}

{\bf Step 7.}
So far we have been treating theorem \ref{thm: NV} as a ``black box'', and we used it to prove an upper bound for the Minkowski content of $\mathcal{Z}$.  It turns out that a more careful look at the proof of theorem \ref{thm: NV} also renders a rectifiable map for $\mathcal{Z}$, hence it concludes the proof of theorem \ref{thm: covering argument}.

Another way to show the rectifiability of $\mathcal{Z}$ without opening the ``black box'' is to cite the following theorem of Azzam and Tolsa. 
\footnote{
As Aaron Naber kindly pointed out to the author, this argument could be misleading, because it actually takes an unnecessary detour when all the proofs are unfolded. Nevertheless, it may serve the readers who want to verify the result and are willing to take the established theorems for granted.}

\begin{Theorem}[\cite{azzam2015characterization}, Corollary 1.3] \label{thm: AT}
Assume $S\subset B(2)$ is a $\mathcal{H}_2$-measurable set and has finite Hausdorff measure, let $\lambda$ be the restriction of $\mathcal{H}_2$ to $S$. Assume that for $\lambda$-a.e. $z$, 
$$ \int_0^1 \frac{D_{\lambda}^2(z,s)}{s} \,ds  <+\infty,
$$
then $S$ is $2$-rectifiable.
\end{Theorem}

Now invoke theorem \ref{thm: AT} and let $S$ be the set $\mathcal{Z}$.
By \eqref{eqn: bound on measure},

\begin{align*}
\int_{B(1)} \int_0^1 \frac{D_\lambda^2(z,s)}{s}\,ds \,d\lambda(z) 
\le & 
C\int_{B(1)}\int_0^1\int_{B_z(s)}\frac{W_{2s}^{32s}(p)}{s^3}\,d\lambda(p)\,ds\,d\lambda(z)
\\
\le & C\int_{B(2)}\int_0^1\int_{B_p(s)} \frac{W_{2s}^{32s}(p)}{s^3}\,d\lambda(z)\,ds\,d\lambda(p)
\\
\le & C K' \int_{B(2)}\int_0^1 \frac{W_{2s}^{32s}(p)}{s}\,ds\,d\lambda(p)
\end{align*}
The same estimate as \eqref{eqn: bound int W/t} gives
$$\int_0^1 \frac{W_{2s}^{32s}(p)}{s}\,ds\le \ln(16)\Lambda.$$
Thus
$$
CK' \int_{B(2)}\int_0^1 \frac{W_{2s}^{32s}(p)}{s}\,ds\,d\lambda(p)\le 4C(K')^2\ln(16)\Lambda < \infty.
$$
Therefore, the conditions of theorem \ref{thm: AT} are satisfied for $\mathcal{Z}\cap B(1)$, hence $\mathcal{Z}\cap B(1)$ is a rectifiable set, and the result is proved.
\end{proof}

\begin{proof}[Proof of theorem \ref{thm: main}]
Let $R_0$ be the constant given by proposition \ref{p: estimate of D}. Cover $  B_{x_0}(R)$ by finitely many Euclidean balls of radius $R_0/32$. Let $  B_{x_i}(R_0/32)$ be such a ball,
we claim that there exists a constant $C$ such that
$$
\mathcal{I}(x,r)=N_\phi(x,r)+Cr^2
$$
is a taming function for $Z\cap B_{x_i}(R_0/16)$ on the ball $B_{x_i}(R_0/16)$.

In fact, it follows from the definition that $N_\phi(x,r)$ is non-negative and continuous. By equation \eqref{eqn: r der of N}, there exists $C_1>0$ such that $\mathcal{I}_1(x,r)=N_\phi(x,r)+C_1r^2$ is increasing in $r$. By proposition \ref{p: estimate of D}, there exists $C_2$, such that for $\mathcal{I}_2(x,r)=\mathcal{I}_1(x,r)+C_2r^2$, one has
$$
D_\mu^2(x,r)\le \frac{C_1}{r^2} \int_{  B_x(r)} [\mathcal{I}_2(32r)-\mathcal{I}_2(2r)]d\mu(x)
$$
for every Radon measure supported in $Z\cap   B_{x_i}(R_0)$ and $r\le 8R_0$, thus $\mathcal{I}_2$ satisfies condition (3) of definition \ref{def: tame}.

Notice that since $\mathcal{I}_1(x,r)$ is increasing in $r$, for $\tilde\beta>0$, the inequality
$$
\mathcal{I}_2(x,2r)-\mathcal{I}_2(x,\tilde\beta r)<\delta
$$
implies that $r<\sqrt{\delta/(4C_2)}$. Therefore, lemma \ref{l: cor of compactness 2} implies $\mathcal{I}_2$ satisfies condition (\ref{con: pinching}) of definition \ref{def: tame}.

In conclusion, $\mathcal{I}_2(x,r)$ is a taming function for $Z$ on $B_{x_i}(R_0/16)$, therefore theorem \ref{thm: main} follows from theorem \ref{thm: covering argument}.
\end{proof}

\bibliographystyle{amsalpha}
\bibliography{references}

\end{document}